\documentclass[11pt,a4paper]{amsart}
\usepackage{a4wide}
\pdfoutput=1
\usepackage{graphicx}
\usepackage{enumitem}  
\usepackage{amssymb}
\usepackage{amsmath}
\usepackage{color}
\usepackage{bbm}
\usepackage{pifont}
\usepackage{array} 
\usepackage{mathtools} 
\usepackage{kbordermatrix}

\usepackage[pdftex]{hyperref}
\newcommand{\ie}{\textit{i.e.},~} 
\newcommand{\etc}{\textit{etc.}} 
\newcommand{\tick}{\ding{51}}
\newcommand{\cross}{\ding{55}}

\newcommand{\cC}{\mathcal{C}}

\newcommand{\cE}{\mathcal{E}}
\newcommand{\cF}{\mathcal{F}}
\newcommand{\cG}{\mathcal{G}}

\newcommand{\cM}{\mathcal{M}}
\newcommand{\cN}{\mathcal{N}}
\newcommand{\cO}{\mathcal{O}}
\newcommand{\cP}{\mathcal{P}}

\newcommand{\cS}{\mathcal{S}}
\newcommand{\cT}{\mathcal{T}}

\newcommand{\cV}{\mathcal{V}}

\newcommand{\RR}{\mathbb{R}}

\newcommand{\gs}{\sigma}
\newcommand{\ep}{\epsilon}

\newcommand\fF{\cT} 
\DeclareMathOperator{\sew}{Sew}
\def\rd{s}
\def\rr{r}
\newcommand{\Gale}[1]{{#1}^{\star}} 
\newcommand{\pGale}[1]{\left( {#1}\right)^{\star}} 
\def\ci{\cC} 
\def\co{\Gale\ci} 
\def\ve{\cV} 
\def\cov{\Gale\ve} 
\def\IG{\mathrm{IG}}

\newcommand{\veczero}{\boldsymbol{0}}

\newcommand{\conv}{\mathrm{conv}}

\newcommand{\verts}{\mathrm{vert}}

\newcommand{\rank}{\mathrm{rank}}

\newcommand{\sprod}[2]{\langle {#1} , {#2} \rangle}
\newcommand{\defn}[1]{\emph{#1}}
\newcommand{\set}[2]{\ensuremath{\left\{#1\,\middle|\,#2\right\}}} 

\newcommand{\ffloor}[2]{\left\lfloor{\frac{#1}{#2}}\right\rfloor}
\newcommand{\fceil}[2]{\left\lceil {\frac{#1}{#2}} \right\rceil}

\newcommand{\rst}[1]{\ensuremath{{\mathbin\upharpoonright}\raise-.5ex\hbox{$#1$}}} 
\def\restriction#1#2{\mathchoice
              {\setbox1\hbox{${\displaystyle #1}_{\scriptstyle #2}$}
              \restrictionaux{#1}{#2}}
              {\setbox1\hbox{${\textstyle #1}_{\scriptstyle #2}$}
              \restrictionaux{#1}{#2}}
              {\setbox1\hbox{${\scriptstyle #1}_{\scriptscriptstyle #2}$}
              \restrictionaux{#1}{#2}}
              {\setbox1\hbox{${\scriptscriptstyle #1}_{\scriptscriptstyle #2}$}
              \restrictionaux{#1}{#2}}}
\def\restrictionaux#1#2{{#1\,\smash{\vrule height .8\ht1 depth .85\dp1}}_{\,#2}} 
\newcommand{\cyc}[2]{C_{#1}({#2})} 

\newcommand{\lnei}[2]{\operatorname{nb}_l({#1,#2})}
\newcommand{\lpol}[2]{\operatorname{p}_l({#1,#2})}
\newcommand{\nnei}[2]{\operatorname{nb}({#1,#2})}
\newcommand{\lle}[2]{\ell_l({#1,#2})}

\newcommand{\lnr}[2]{\operatorname{nr}_l({#1,#2})}

\def\e{\mathrm{e}}

\newtheorem{theorem}{Theorem}[section] 
\newtheorem{proposition}[theorem]{Proposition} 
\newtheorem{lemma}[theorem]{Lemma} 
\newtheorem{corollary}[theorem]{Corollary}

\newtheorem{question}{Question}

\theoremstyle{remark}

\newtheorem{remark}[theorem]{Remark}
\newtheorem{observation}[theorem]{Observation}
\newtheorem{example}[theorem]{Example}

\theoremstyle{definition}
\newtheorem{definition}[theorem]{Definition} 

\newtheorem{constr}{Construction}

\begin{document}

\title{Many neighborly polytopes and oriented matroids}

\author{Arnau Padrol}

\address{Institut f\" ur Mathematik, FU Berlin, Arnimallee 2, 14195 Berlin, Germany}
\email{arnau.padrol@fu-berlin.de}
\thanks{
This research was supported by the DFG Collaborative Research Center SFB/TR~109 ``Discretization in Geometry and Dynamics'' as well as by AGAUR grant 2009 SGR 1040 and FI-DGR grant from Catalunya's government and the ESF 
}

\begin{abstract}
In this paper we present a new technique to construct neighborly polytopes, and use it to prove a lower bound of $\left.{\left(\left( r+d \right) ^{\left( \frac{r}{2}+\frac{d}{2} \right) ^{2}}\right)}\middle/{\left({r}^{{(\frac{r}{2})}^{2}}{d}^{{(\frac{d}{2})}^{2}}{\e^{3\frac{r}{2}\frac{d}{2}}}\right)}\right.$ for the number of combinatorial types of vertex-labeled neighborly polytopes in even dimension $d$ with $r+d+1$ vertices.
This improves current bounds on the number of combinatorial types of polytopes.

The previous best lower bounds for the number of neighborly polytopes were found by Shemer in 1982 using a technique called the \emph{Sewing Construction}. We provide a simpler proof that sewing works, and generalize it to oriented matroids in two ways: to \emph{Extended Sewing} and to \emph{Gale Sewing}. Our lower bound is obtained by estimating the number of polytopes that can be constructed via Gale Sewing. Combining both new techniques, we are also able to construct many non-realizable neighborly oriented matroids. 
\end{abstract}

\maketitle

\section{Introduction}
\label{intro}
A polytope is said to be \emph{$k$-neighborly} if every subset of vertices of size at most~$k$ is the set of vertices of one of its faces. It is easy to see that if a $d$-polytope is $k$-neighborly for any $k>\ffloor{d}{2}$, then it must be the $d$-dimensional simplex $\Delta_d$. This is why a $d$-polytope is called \emph{neighborly} if it is $\ffloor{d}{2}$-neighborly. Analogously, an (acyclic) oriented matroid of rank $r$ is called \emph{neighborly} if every $\ffloor{r-1}{2}$ elements form a face (see~\cite[Chapter~9]{OrientedMatroids1993}). 

Neighborly polytopes form a very interesting family of polytopes because of their extremal properties. In particular, McMullen's {Upper Bound Theorem}~\cite{McMullen1970} states that the number of $i$-dimensional faces of a $d$-polytope $P$ with $n$ vertices is maximal for simplicial neighborly polytopes, for all $i$.
Any set of $n$ points on the \defn{moment curve} in $\RR^d$, $\{(t,t^2,\dots,t^d):t\in \RR\}$, is the set of vertices of a neighborly polytope.
Since the combinatorial type of this polytope does not depend on the particular choice of points (see~\cite[Section~4.7]{GruenbaumEtal2003}), we denote it as $\cyc{d}{n}$\index{$\cyc{d}{n}$}, the \defn{cyclic polytope} with $n$ vertices in $\RR^d$. 

The first examples of non-cyclic neighborly polytopes were found in 1967 by Gr\"unbaum~\cite[Section 7.2]{GruenbaumEtal2003}.
In 1981, Barnette introduced the \defn{facet splitting} technique~\cite{Barnette1981}, that allowed him to construct infinitely many neighborly polytopes, and to prove that \defn{$\nnei{n}{d}$}, the number of (combinatorial types of) neighborly $d$-polytopes with $n$ vertices, is bigger than \[\nnei{n}{d}\geq\frac{(2n-4)!}{n!(n-2)!\binom{n}{d-3}}\sim4^{n(1+o(1))}.\]
(Here and below, the asymptotic notation $o(1)$ refers to fixed~$d$ and $n\rightarrow \infty$.)
  
This bound was improved by Shemer in~\cite{Shemer1982}, where he introduced the \defn{Sewing Construction} to build an infinite family of neighborly polytopes in any even dimension. 
Given a neighborly $d$-polytope with $n$ vertices and a suitable flag of faces, one can ``sew'' a new vertex onto it to get a new neighborly $d$-polytope with $n+1$ vertices. With this construction, Shemer proved that $\nnei{n}{d}$ is greater than \[\nnei{n}{d}\geq\frac{1}{2}\left(\left(\frac{d}{2}-1\right)\ffloor{n-2}{d+1}\right)!\sim n^{c_d n(1+o(1))},\] where $c_d\rightarrow \frac{1}{2}$ when $d\rightarrow\infty$.

The main result of this paper is the following theorem, proved in Section~\ref{sec:counting}, that provides a new lower bound for $\lnei{n}{d}$, the number of \emph{vertex-labeled} combinatorial types of neighborly polytopes with $n$ vertices and dimension~$d$.

\medskip
\noindent\textbf{Theorem \ref{thm:lblnei}}
\emph{The number of labeled neighborly polytopes in even dimension $d$ with $r+d+1$ vertices fulfills}
\begin{equation}\label{eq:thebound}
\tag{$\bigstar$}
\lnei{r+d+1}{d}\geq \frac{\left( r+d \right) ^{\left( \frac{r}{2}+\frac{d}{2} \right) ^{2}}}{{r}^{{(\frac{r}{2})}^{2}}{d}^{{(\frac{d}{2})}^{2}}{{\e}^{3\frac{r}{2}\frac{d}{2}}}}. \end{equation}

\medskip
This bound is always greater than 
\[
    \lnei{n}{d}\geq  \left( \frac{n-1}{\e^{3/2}}\right)^{\frac12 d(n-d-1)}\sim n^{\frac{dn}{2}(1+o(1))},
\]
and dividing by $n!$  easily shows this to improve Shemer's bound also in the unlabeled case. Moreover, when $d$ is odd we can use the bound $\lnei{r+d+1}{d}\geq \lnei{r+d}{d-1}$, which follows by taking pyramids (cf. Corollary~\ref{cor:oddneighpoly}).

Of course, \eqref{eq:thebound} is also a lower bound for \defn{$\lpol{n}{d}$}\index{$\lpol{n}{d}$}, the number of combinatorial types of vertex-labeled $d$-polytopes with $n$ vertices, and is even greater than
\[\lpol{n}{d}\geq \left(\frac{n-d}{d}\right)^{\frac{nd}{4}},\]
which is, as far as the author knows, the current best lower bound for $\lpol{n}{d}$ (valid only for $n\geq 2d$). This bound was found by Alon in 1986~\cite{Alon1986}.

\begin{remark}
To the best of the author's knowledge, the only known upper bounds for $\lnei{n}{d}$ are the upper bounds for $\lpol{n}{d}$. 
Alon proved in~\cite{Alon1986} that 
\[\lpol{n}{d}\leq \left(\frac{n}{d} \right)^{d^2n(1+o(1))}\text{ when }\tfrac{n}{d}\rightarrow \infty.\]
improving a similar bound for simplicial polytopes due to
Goodman and Pollack~\cite{GoodmanPollack1986}
\end{remark}

\medskip
We can summarize the main contributions of this paper as follows.

\begin{enumerate}
\item First, we show that Shemer's Sewing Construction can be very transparently explained (and generalized) in terms of \emph{lexicographic extensions} of oriented matroids (Section~\ref{sec:shemer}). In fact, the same framework also explains Lee \& Menzel's related construction of $A$-sewing for non-simplicial polytopes~\cite{LeeMenzel2010} (Observation~\ref{obs:Asewing}), and the results in~\cite{TrelfordVigh2011} on faces of sewn polytopes. Moreover, it naturally applies also to odd dimension just like Bistriczky's version of the Sewing Theorem~\cite{Bisztriczky2000}.

\item Next, we introduce two new construction techniques for polytopes. The first, \defn{Extended Sewing} (Construction~\ref{constr:cE}) is based on our Extended Sewing Theorem~\ref{thm:extshemersewing}. 
It is a generalization of Shemer's sewing to oriented matroids that is valid for any rank and works for a large family of flags of faces (suggested in \cite[Remark 7.4]{Shemer1982}), including the ones obtained by Barnette's facet splitting~\cite{Barnette1981}. 
Moreover, Extended Sewing is optimal in the sense that in odd ranks, the flags of faces constructed in this way are the only ones that yield neighborly polytopes (Proposition~\ref{prop:uniqueflags}).

\item Our second (and most important) new technique is \emph{Gale Sewing} (Construction~\ref{constr:cG}), whose key ingredient is the Double Extension Theorem~\ref{thm:thethm}. It lexicographically extends \emph{duals} of neighborly polytopes and oriented matroids.  With it, we construct a large family of polytopes called $\cG$.
This family contains all the neighborly polytopes constructed in~\cite{Devyatov2011}, which arise as a special case of Gale Sewing for polytopes of corank~$3$.

\item Using Extended Sewing, we construct three families of neighborly polytopes --- $\cS$, $\cE$ and $\cO$ --- the largest of which is $\cO$. In Section~\ref{sec:comparing}, we show that $\cO\subseteq \cG$ (Corollary~\ref{cor:cOsubsetcG}), and in this sense, Gale Sewing is a generalization of Extended Sewing. However, it is not true that the Double Extension Theorem~\ref{thm:thethm} generalizes the Extended Sewing Theorem~\ref{thm:extshemersewing} (cf. Remark~\ref{rmk:doesnotgeneralize}).

\item The bound \eqref{eq:thebound} is obtained in Theorem~\ref{thm:lblnei} by estimating the number of different polytopes in $\cG$. 
\item To tie our constructions together, we show that combining Extended Sewing and Gale Sewing yields non-realizable neighborly oriented matroids with $n$~vertices and rank~$\rd$ for any $\rd\geq 5$ and $n\geq \rd+5$ (Theorem~\ref{thm:nonrealizable}). Even more, in Theorem~\ref{thm:nonrealizablebound} we show that lower bounds proportional to \eqref{eq:thebound} also hold for the number of labeled non-realizable neighborly oriented matroids.
\end{enumerate}

\begin{observation}
Sanyal and Ziegler proved  that the number 
of neighborly simplicial $(d - 2)$-polytopes on $n - 1$ vertices is a lower bound for the number of $d$-dimensional neighborly cubical polytopes with $2^n$ vertices~\cite[Corollary 3.8]{SanyalZiegler2010}. Hence, \eqref{eq:thebound} also yields lower bounds the number of neighborly cubical polytopes.
\end{observation}

\begin{observation}
It can be proven that all the polytopes that belong to $\cG$ are inscribable, that is, that they can be realized with all their vertices on a sphere~\cite{GonskaPadrol}. Hence, \eqref{eq:thebound} is also valid as a lower bound for the number of inscribable neighborly polytopes and for the number of neighborly Delaunay triangulations 
(see also Remark~\ref{rmk:inscribable}).
\end{observation}

We present our results after the introductory Section~\ref{sec:defs}, which may be skimmed with the exception of the statement of Proposition~\ref{prop:allquotientsofle}.
The proof of this and some smaller results are relegated to Appendix~\ref{sec:appendix} so as not to interrupt the flow of reading. The presentation of Extended Sewing and Gale Sewing is mostly independent, and hence a reader interested only in the the proof of the lower bound~\eqref{eq:thebound} can skip Sections~\ref{sec:shemer} and~\ref{sec:comparing} and concentrate on Sections~\ref{sec:thethm} and~\ref{sec:counting}.

\section{Neighborly and balanced Oriented Matroids}
\label{sec:defs}
We assume that the reader has some familiarity with the basics of oriented matroid theory; we refer to~\cite{OrientedMatroids1993} for a comprehensive reference. 
\subsection{Preliminaries}

As for notation, $\cM$ will be an oriented matroid of rank $\rd$ on a ground set~$E$, with circuits $\ci(\cM)$, cocircuits $\co(\cM)$, vectors $\ve(\cM)$ and covectors $\cov(\cM)$. Its dual $\Gale\cM$ has rank $r=n-\rd$. $\cM$ is \defn{uniform} if the underlying matroid $\underline \cM$ is uniform, that is, every subset of size $\rd$ is a basis.

We view every vector/covector $X$ of $\cM$ as a function from $E$ to $\{+,-,0\}$ (or to $\{\pm 1,0\}$). Hence, we will say $X(e)=+$ or $X(e)>0$. The \defn{support} $\underline X\subset E$ of a vector/covector $X$ is $\underline X=\set{e\in E}{X(e)\neq 0}$, and we say that a vector $X$ is \defn{positive} if $X(e)\geq 0$ for all $e\in E$. 

We say that two oriented matroids $\cM_1$ and $\cM_2$ on respective ground sets $E_1$ and $E_2$ are \defn{isomorphic}, $\cM_1\simeq\cM_2$, when there is a bijection between $E_1$ and $E_2$ that sends circuits of $\cM_1$ to circuits of $\cM_2$ (and equivalently for vectors, cocircuits or covectors) in such a way that the signs are preserved.

A matroid 
$\cM$ is \emph{acyclic} if the whole ground set is the support of a positive covector. Its \emph{facets} are the complements of the supports of its positive cocircuits, and its \emph{faces} the complements of its positive covectors. Faces of rank $1$ are called \emph{vertices} of $\cM$. In particular, every $d$-polytope is an acyclic matroid of rank $d+1$.
Similarly, a matroid is \emph{totally cyclic} if the whole ground set is the support of a positive vector. 

We will need some constructions to deal with an oriented matroid $\cM$, in particular the \emph{deletion} $\cM\setminus e$ and the \emph{contraction} $\cM/e$ of an element $e$. They are defined by their covectors (by $\restriction{C}{E\setminus\{e\}}$ we denote the restriction of $C$ to $E\setminus\{e\}$):
\begin{align*}
 \cov(\cM\setminus e)&=\set{\restriction{C}{E\setminus\{e\}}}{C\in \cov(\cM)},\\
 \cov(\cM\,/\,e)&=\set{\restriction{C}{E\setminus\{e\}}}{C\in \cov(\cM)\text{ such that }C(e)=0}.
\end{align*}

Deletion and contraction are dual operations ---\, $\pGale{\cM\setminus e}=\left(\Gale{\cM}/e\right)$ \,--- that commute ---\, $\left(\cM\setminus p\right)/q= \left(\cM/q\right)\setminus p$ \,--- and naturally extend to subsets $S\subseteq E$ by iteratively deleting (resp. contracting) every element in $S$.\\

To illustrate our results, we use \emph{affine Gale diagrams}, which are described in detail in~\cite[Chapter~6]{LecturesOnPolytopes1995} or~\cite{Sturmfels1988b}. They turn a labeled vector configuration $V=\{v_1,\dots,v_n\}\subset\RR^r$ (for simplicity we assume that no $v_i$ is $\veczero$) into a labeled affine point configuration $A=\{a_1,\dots,a_n\}\subset\RR^{r-1}$. For this, take a vector $c\in\RR^{r}$ such that $\sprod{v_i}{c}$ is not $0$ for any $v_i$ (here $\sprod{}{}$ denotes the standard scalar product). Then $A$ is the point configuration in the hyperplane with equation $\sprod{x}{c}=1$ consisting of the points $a_i:=\frac{v_i}{\sprod{v_i}{c}}$ for $v_i\in V$. We call $a_i$ a \defn{positive point} if $\sprod{v_i}{c}>0$, and a \defn{negative point} if $\sprod{a_i}{c}<0$. In our figures, positive points are depicted as full circles and negative points are empty circles. See the example of Figure~\ref{fig:affineGale}.

\begin{figure}[htpb]
\centering
\includegraphics[width=.45\linewidth]{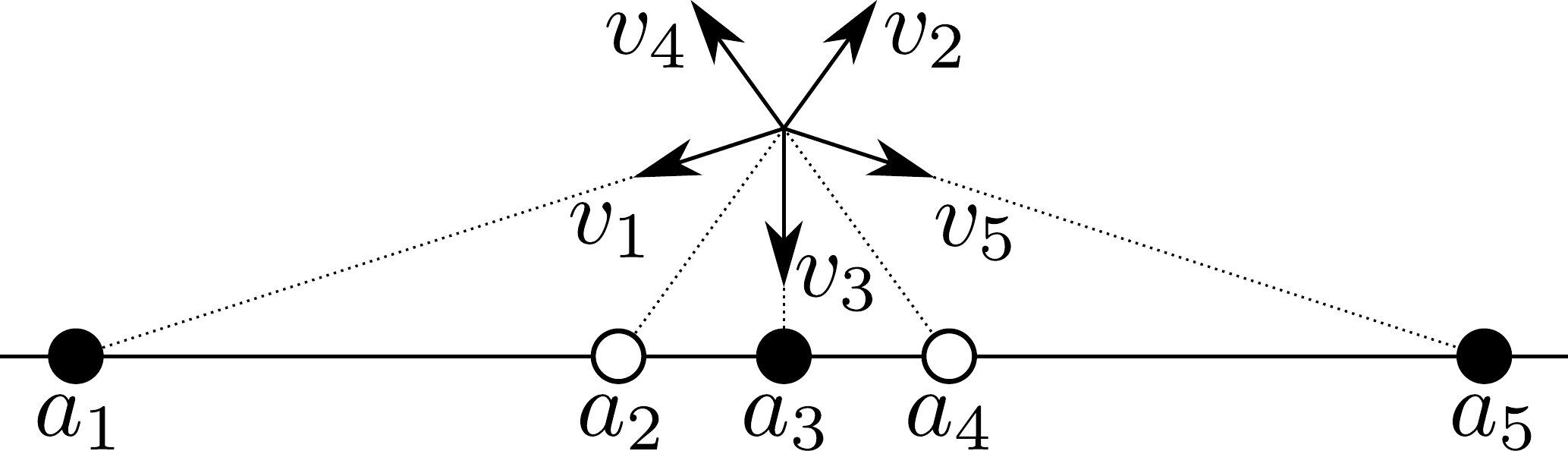}
\caption{An affine Gale diagram in $\RR^1$ from a vector configuration in $\RR^2$.}
\label{fig:affineGale}
\end{figure}

\subsection{Neighborly and balanced oriented matroids}

As we have already mentioned, neighborliness is a purely combinatorial concept that can be easily defined in terms of oriented matroids. 
\begin{definition}\label{def:neigh}
An oriented matroid $\cM$ of rank $\rd$ on a ground set~$E$ is \emph{neighborly} if
every subset $S\subset E$ of size at most $\ffloor{\rd-1}{2}$ is a face of~$\cM$. That is, there exists a covector $C\in\co(\cM)$ with $C(e)=0$ for $e\in S$ and $C(e)=+$ otherwise. 
\end{definition}

Thus, realizable neighborly oriented matroids correspond to neighborly polytopes. However, not all neighborly oriented matroids are realizable (see Section~\ref{sec:nonrealizable}).
Nevertheless, several properties of neighborly polytopes extend to all neighborly oriented matroids (cf.~\cite{CordovilDuchet2000} and~\cite{Sturmfels1988}).

An important property of neighborly matroids of odd rank (in the realizable case, neighborly polytopes of even dimension) is that they are rigid. We call an oriented matroid \emph{rigid} if there is no other oriented matroid that has its face lattice; equivalently, if the face lattice determines its whole set of covectors. This result was first discovered by Shemer for neighborly polytopes~\cite{Shemer1982} and later extended to all neighborly oriented matroids  by Sturmfels~\cite{Sturmfels1988}.

\begin{theorem}[{\cite[Theorem 4.2]{Sturmfels1988}}]\label{thm:neigharerigid}
 Every neighborly oriented matroid of odd rank is rigid.
\end{theorem}
\smallskip

Definition~\ref{def:neigh} is based on the presentation by cocircuits, but neighborly matroids can also be characterized by their circuits. Said differently, one can characterize dual-to-neighborly matroids in terms of cocircuits. These are balanced matroids.

\begin{definition}\label{def:balanced}
An oriented matroid $\cM$ of rank $\rr$ and $n$ elements is \defn{balanced} if every cocircuit~$C$ of $\cM$ is balanced; and a cocircuit~$C\in\co(\cM)$ is \defn{balanced} when 
 \[\ffloor{n-\rr+1}{2}\leq |C^+|\leq\fceil{n-\rr+1}{2}.\]
where $C^+=\set{e\in E}{C(e)=+}$.
\end{definition}
These cocircuits (and matroids) are called balanced because of the fact that in a uniform oriented matroid, a cocircuit is balanced if and only if it has the same number of positive and negative elements ($\pm 1$ if the corank is odd).

That neighborliness and balancedness are dual concepts is already implicit in the work of Gale~\cite{Gale1963} for the case of polytopes, and one can find a proof for oriented matroids by Sturmfels in~\cite{Sturmfels1988}.

\begin{proposition}[{\cite[Proposition 3.2]{Sturmfels1988}}]
 An oriented matroid $\cM$ is neighborly if and only if its dual matroid $\Gale{\cM}$ is balanced.
\end{proposition}

\subsection{Single element extensions}

Let $\cM$ be an oriented matroid on a ground set $E$. A \emph{single element extension} of~$\cM$ by an element $p$ is an oriented matroid $\tilde \cM$ on the ground set $E\cup \{p\}$ for some $p\notin E$, such that $\cM$ is the deletion $\tilde\cM\setminus p$.
We will only consider extensions that do not increase the rank, \ie $\rank(\tilde \cM)=\rank(\cM)$.

A concept crucial to understanding a single element extension of $\cM$ is its signature, which we define in the following proposition (cf.~\cite[Proposition 7.1.4]{OrientedMatroids1993}). 

\begin{proposition}(\cite[Proposition 7.1.4]{OrientedMatroids1993},\cite{LasVergnas1978})
Let $\tilde \cM$ be a single element extension of~$\cM$ by $p$. Then, for every cocircuit $C\in\co(\cM)$, there is a unique way to extend~$C$ to a cocircuit of~$\tilde \cM$. 

That is, there is a unique function $\gs$ from $\co(\cM)\rightarrow \{+,-,0\}$ such that for each $C\in \co(\cM)$ there is a cocircuit $C'\in\co(\tilde\cM)$ with $C'(p)=\gs(C)$ and $C'(e)=C(e)$ for $e\in E$. The function $\gs$ is called the \emph{signature} of the extension.

Moreover, the signature $\gs$ uniquely determines the oriented matroid $\tilde\cM$.
\end{proposition}

Although not every map from $\co(\cM)$ to $\{0,+,-\}$ corresponds to the signature of an extension (see~\cite[Proposition 7.1.8]{OrientedMatroids1993}), we will only work with one specific  family of single element extensions called \emph{lexicographic extensions}.

\begin{definition}\label{def:le}
Let $\cM$ be a rank~$\rr$ oriented matroid on a ground set~$E$. Let $(a_1,a_2,\dots, a_k)$ be an ordered subset of $E$ and let $(\ep_1,\ep_2,\dots,\ep_k)\in\{+,-\}^k$ be a sign vector. The \emph{lexicographic extension} $\cM[p]$ of $\cM$ by $p=[a_1^{\ep_1},a_2^{\ep_2},\dots,a_k^{\ep_k}]$ is the oriented matroid on the ground set $E\cup \{p\}$ which is the single element extension of $\cM$ whose signature $\sigma:\co(\cM)\rightarrow \{+,-,0\}$ maps $C\in \co(\cM)$ to
\begin{equation*}
\gs(C)\mapsto 
\begin{cases} \ep_iC({a_i})& \text{if $i$ is minimal with $C({a_i})\neq 0$,}
\\
0&\text{if $C({a_i})=0$ for $i=1,\dots,k$.}
\end{cases}
\end{equation*}
We will also use $\cM[a_1^{\ep_1},\dots,a_k^{\ep_k}]$ to denote the lexicographic extension $\cM[p]$ of $\cM$ by $p=[a_1^{\ep_1},\dots,a_k^{\ep_k}]$.
\end{definition}

\begin{remark}
If $\cM$ is a uniform matroid of rank $\rr$, then $\cM[a_1^{\ep_1},\dots,a_k^{\ep_k}]$ is uniform if and only if $k\geq \rr$. In this situation, the $a_i^{\ep_i}$ with $i>\rr$ are irrelevant, so we can assume that $k=\rr$.
This is the most interesting case for us.
\end{remark}

An important property is that lexicographic extensions preserve realizability (cf. \cite[Section 7.2]{OrientedMatroids1993}).

\begin{lemma}\label{lem:realizablele}
$\cM[p]$ is realizable if and only if $\cM$ is realizable.
\end{lemma}
In the setting of a vector configuration $V$, the lexicographic extension by $p=[a_1^{\ep_1},a_2^{\ep_2},\dots,a_k^{\ep_k}]$ is very easy to understand.  For every hyperplane ${ H}$ spanned by vectors in $V\setminus\{a_1\}$, the new vector $p$ must lie on the same side as $\ep_1a_1$; for hyperplanes containing $a_1$ but not $a_2$, $p$ must lie on the same side as $\ep_2a_2$; \etc\  
This is clearly achieved by the vector $p=\ep_1a_1+\delta \ep_2a_2+\delta^2 \ep_3a_3+\dots+\delta^{k-1} \ep_ka_k$ for some $\delta>0$ small enough. Equivalently, a suitable $p$ can be found by placing a new vector on top of $\ep_1 a_1$, then perturbing it slightly towards $\ep_2 a_2$, then towards $\ep_3a_3$ and so on. See Figure~\ref{fig:le} for an example of this procedure on an affine diagram. 

\begin{figure}[htpb]
\begin{center}
\includegraphics[width=.5\linewidth]{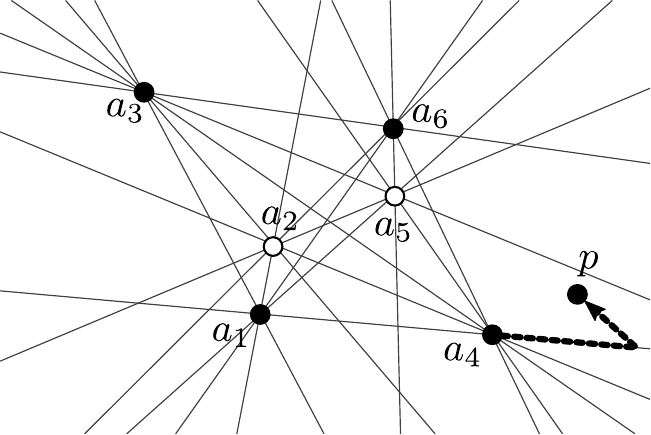}
\end{center}
\caption{An affine Gale diagram, and its lexicographic extension by $p=[a_4^+,a_1^-,a_6^+]$.}\label{fig:le}
\end{figure}

Lexicographic extensions on uniform matroids behave well with respect to contractions. 
The upcoming Proposition~\ref{prop:allquotientsofle} can be used to iteratively explain all cocircuits of a lexicographic extension, and hence can be seen as the restriction of~\cite[Proposition 7.1.4]{OrientedMatroids1993} to lexicographic extensions. It is a very useful tool that will be used extensively. Its proof is not complicated and can be found in Appendix~\ref{sec:appendix}.

\begin{proposition}\label{prop:allquotientsofle}
Let $\cM$ be a uniform oriented matroid of rank $\rr$ on a ground set~$E$, and let 
$\cM[p]$ be the lexicographic extension of $\cM$ by $p=[a_1^{\ep_1},a_2^{\ep_2},\dots,a_\rr^{\ep_\rr}]$.
Then
\begin{align}
\cM[p]/p\ &\stackrel{\varphi}{\simeq}\ (\cM/a_1)[a_2^{-\ep_1\ep_2},\dots,a_\rr^{-\ep_1\ep_\rr}],\label{eq:Mmodp}\\
\cM[p]/a_i\ &=\ (\cM/a_i)[a_1^{\ep_1},\dots,a_{i-1}^{\ep_{i-1}},a_{i+1}^{\ep_{i+1}},\dots,a_\rr^{\ep_\rr}], \text{ and }\label{eq:Mmoda}\\
\cM[p]/e\ &=\ (\cM/e)[a_1^{\ep_1},a_2^{\ep_2},\dots,a_{\rr-1}^{\ep_{\rr-1}}];\label{eq:Mmode}
\end{align}
where $e\in E$ is any element different from $p$ and any $a_i$. The isomorphism $\varphi$ in~\eqref{eq:Mmodp} is $\varphi(e)=e$ for all $e\in E\setminus \{p,a_1\}$ and $\varphi(a_1)=[a_2^{-\ep_1\ep_2},\dots,a_\rr^{-\ep_1\ep_\rr}]$; where the latter is the extending element.
\end{proposition}

The most interesting case is \eqref{eq:Mmodp}. If $\cM$ is realized by~$V$ and $V\cup \{p\}$ realizes the lexicographic extension of $\cM$ by $p=[a_1^{\ep_1},a_2^{\ep_2},\dots,a_\rr^{\ep_\rr}]$, then the intuition behind the isomorphism $\cM[p]/p\simeq \cM/a_1[a_2^{-\ep_1\ep_2},\dots,a_\rr^{-\ep_1\ep_\rr}]$ is that every hyperplane spanned by $V$ that goes through $p$ and not through $a_1$ looks very much like some hyperplane that goes through~$a_1$ and not through $p$. 
If $\ep_1=+$, then $a_1$ and~$p$ are very close, which means that when we perturb a hyperplane ${H}$ with $p$ in ${H}^+$ that is spanned by $a_1\cup S$ to its analogue ${H}'$ spanned by $p\cup S$, then $a_1$ lies in~${{H}'}^-$ and the remaining elements are on the same side of ${H}'$ as they were of  ${H}$. 
On the other hand, if $\ep_1=-$, then $a_1$ and~$-p$ are very close, and to perturb ${H}$ to ${H}'$, one must also switch the sign of $a_1$. Hence if $p$ was in ${H}^+$, then $a_1$ is in ${{H}'}^-$.

\section{The Sewing Construction}\label{sec:shemer}

This section is devoted to explaining the Sewing Construction, introduced by Shemer in~\cite{Shemer1982}, that allows to construct an infinite class of neighborly polytopes.
Even if Shemer described it in terms of Gr\"unbaum's \emph{beneath-beyond} technique, it is in fact a lexicographic extension, and we will explain it in these terms. In this section, we use the letter $\cP$ for oriented matroids to reinforce the idea that all the following results translate directly to polytopes.

\subsection{Sewing a point onto a flag}

Let $\cP$ be an acyclic oriented matroid on a ground set $E$, and let $F\subset E$ be a facet of~$\cP$. That is, there exists a cocircuit $C_{F}$ of $\cP$ such that $C_{ F}(e)=0$ if $e\in F$ and $C_{F}(e)=+$ otherwise. Consider a single element extension of $\cP$ by $p$ with signature~$\gs_p$. We say that $p$ is 
\emph{beneath}~$ F$ if $\gs_p(C_{ F})=+$, that $p$ is \emph{beyond}~$ F$ when $\gs_p(C_{ F})=-$, and that $p$ is \emph{on} $ F$ if $\gs_p(C_{ F})=0$.
We say that $p$ lies \emph{exactly beyond} a set of facets $\fF$ if it lies beyond all facets in $\fF$ and beneath all facets not in $\fF$. 

\begin{lemma}[{\cite[Proposition 9.2.2]{OrientedMatroids1993}}]\label{lem:benbey} Let $\tilde \cP$ be a single element extension of $\cP$ with signature $\sigma$. Then the values of $\sigma$ on the facet cocircuits of $\cP$ determine the whole face lattice of $\tilde\cP$.
\end{lemma}

A \emph{flag} of $\cP$ is a strictly increasing sequence of proper faces $F_1\subset F_2 \subset \dots \subset F_k$. We say that a flag $\cF$ is a \emph{subflag} of $\cF'$ if each face~$F$ that belongs to $\cF$ also belongs to $\cF'$.
 Given a flag $\cF=\{F_j\}_{j=1}^k$ of $\cP$, let $\fF_j$ be the set of facets of $\cP$ that contain $F_j$, and let $\sew(\cF):=\fF_1\setminus(\fF_2\setminus(\dots\setminus\fF_k)\dots)$, so that
\begin{equation*}
\sew(\cF)= 
\begin{cases} (\fF_1\setminus\fF_2) \cup (\fF_3\setminus\fF_4)\cup \dots \cup (\fF_{k-1}\setminus\fF_k)& \text{if $k$ is even,}
\\
(\fF_1\setminus\fF_2) \cup (\fF_3\setminus\fF_4)\cup \dots \cup \fF_k &\text{if $k$ is odd.}
\end{cases}
\end{equation*}

Given a polytope $P$ with a flag of faces $\cF=F_1\subset F_2 \subset \dots \subset F_k$, Shemer proved that there always exists an extension exactly beyond $\sew(\cF)$ (\cite[Lemma 4.4]{Shemer1982}), and called this extension sewing onto the flag. We will show that there is a lexicographic extension that realizes the desired signature.

\begin{definition}[Sewing onto a flag]\label{def:sewing}
Let $\cF=\{F_j\}_{j=1}^k$ be a flag of an acyclic matroid $\cP$ on a ground set $E$. We extend the flag with $F_{k+1}=E$ and define $U_j=F_j\setminus F_{j-1}$. 
We say that $p$ is \emph{sewn} onto $\cP$ through~$\cF$, if $\cP[p]$ is a lexicographic extension of $\cP$ by
\[p=[F_1^+,U_2^-,U_3^+,\dots,U_{k+1}^{(-1)^{k}}],\]
where these sets represent their elements in any order. 
Put differently, the lexicographic extension by $p$ is defined by $p=[a_1^{\ep_1},a_2^{\ep_2},\dots,a_{n}^{\ep_{n}}]$, where $a_1, \dots, a_{n}$ are the elements in $F_{k+1}=E$ sorted such that
\begin{itemize}
\item if there is some $m$ such that $a_i\in F_m$ and $a_j\notin F_m$, then $i<j$;
\item if the smallest $m$ such that $a_j\in F_m$ is odd, then $\ep_j=+$; and $\ep_j=-$ otherwise.
\end{itemize}
We use the notation $\cP[\cF]$ to designate the extension $\cP[p]$ when $p$ is sewn onto~$\cP$ through~$\cF$.
\end{definition}

For example, if $\cP$ has $6$ elements and rank $5$, and $F_1=\{a_1,a_2\}$ and $F_2=\{a_1,a_2,a_3,a_4\}$ are the elements of two faces of $\cP$, then the lexicographic extensions by $[a_1^+,a_2^+,a_3^-,a_4^-,a_5^+]$, $[a_2^+,a_1^+,a_3^-,a_4^-,a_6^+]$ or $[a_2^+,a_1^+,a_4^-,a_3^-,a_6^+]$ are extensions by an element sewn through~the flag $F_1\subset F_2$ (note how the orders in the faces and last element of the extension can be chosen arbitrarily). Another example is shown in Figure~\ref{fig:sewingonflags}.

\begin{figure}[htpb]
\centering

\begin{tabular}{m{.16\textwidth}m{.16\textwidth}m{.16\textwidth}m{.16\textwidth}m{.16\textwidth}}
 \includegraphics[width=.16\textwidth]{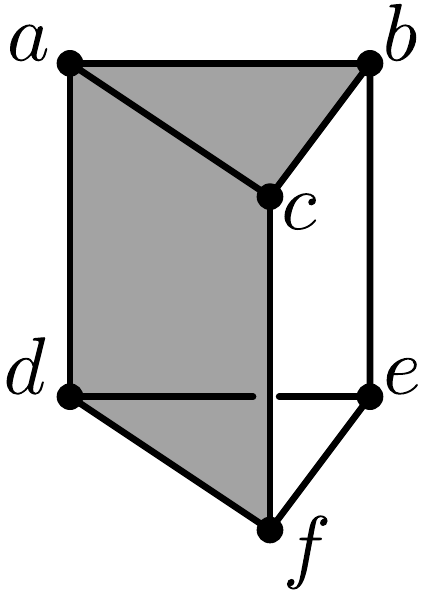}&
 \includegraphics[width=.16\textwidth]{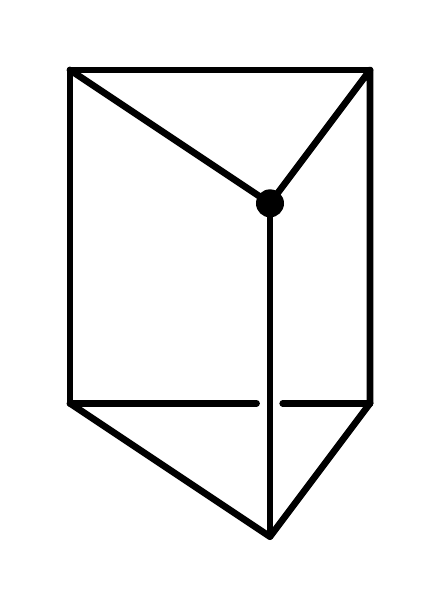}&
 \includegraphics[width=.16\textwidth]{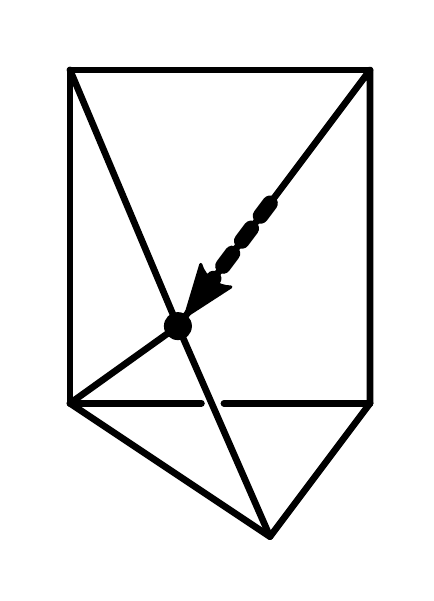}&
 \includegraphics[width=.16\textwidth]{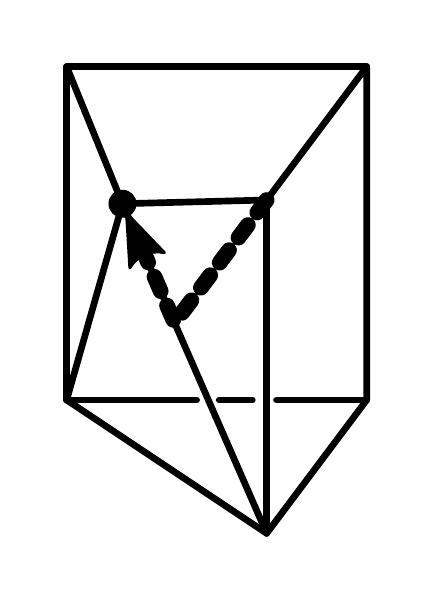}&
 \includegraphics[width=.16\textwidth]{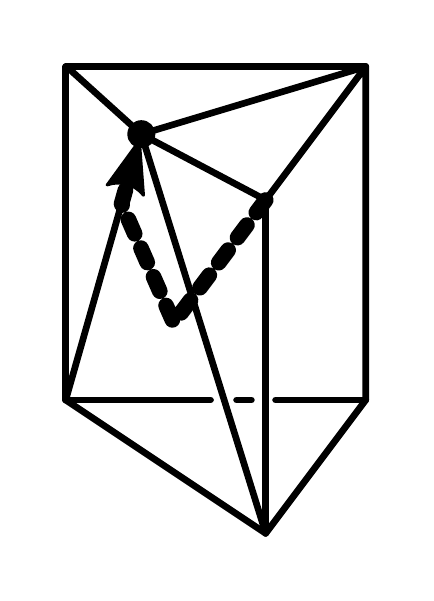}\\
\centering$P\centering$&\centering$P[c^+]\centering$&\centering$P[c^+,b^-]\centering$&\centering$P[c^+,b^-,a^+]\centering$&\centering$P[c^+,b^-,a^+,d^-]\centering$
\end{tabular}
\caption{A polytope $P=\conv\{a,b,c,d,e\}$. Sewing onto the flag $\cF=\{c\}\subseteq\{c,b\}\subseteq \{c,b,a\}$. Shaded facets in $P$ correspond to $\sew(\cF)$.}
 \label{fig:sewingonflags}
\end{figure}

In terms of oriented matroids, the definition of $\cP[\cF]$ is ambiguous, since it can represent different oriented matroids. However, the following proposition (together with Lemma~\ref{lem:benbey}) shows that all the extensions $\cP[\cF]$ have the same face lattice. In particular, this implies that there is no ambiguity when $\cP[\cF]$ is neighborly of odd rank, because these are rigid (Theorem~\ref{thm:neigharerigid}).

\begin{proposition}\label{prop:exactlybeyondle}
Let $\cF=\{F_j\}_{j=1}^k$ be a flag of an acyclic oriented matroid~$\cP$. If $\cP[p]$ is the lexicographic extension $\cP[\cF]$, then $p$ lies exactly beyond $\sew(\cF)$.
\end{proposition}
\begin{proof}
Let the lexicographic extension be by $p=[a_1^{\ep_1},a_2^{\ep_2},\dots,a_{n}^{\ep_{n}}]$ with the elements and signs as in Definition~\ref{def:sewing}.
We have to see that, for $1\leq j\leq k$, $p$ lies beneath any facet in $\fF_{j}\setminus\fF_{j+1}$ if $j$ is even, and beyond any facet in $\fF_{j}\setminus\fF_{j+1}$ if $j$ is odd (with the convention $\fF_{k+1}=\emptyset$).

That is, if $\gs$ is the signature of the lexicographic extension and $F$ a facet of $\cP$ defined by a cocircuit $C_F$, we want to see that 
\[\gs(C_F)=\begin{cases}
            +&\text{ if there is an even $j$ such that $F_{j}\subseteq F$ but $F_{j+1}\not\subseteq F$,}\\
            -&\text{ if there is an odd $j$ such that $F_{j}\subseteq F$ but $F_{j+1}\not\subseteq F$;}
           \end{cases}\]
where $F_{k+1}=E$, the ground set of $\cP$.
 
In our case, if $F$ is in $\fF_{j}\setminus\fF_{j+1}$ then the first $a_i$ with $C_F(a_i)\neq 0$ belongs to~$F_{j+1}$ and thus $\ep_i=+$ if $j$ is even and $\ep_i=-$ if $j$ is odd. Therefore, since by definition of lexicographic extension $\gs(C_F)=\ep_iC_F(a_i)=\ep_i$, then $\gs(C_F)=+$ (\ie $p$ is beneath $F$) when $j$ is even while $\gs(C_F)=-$ (\ie $p$ is beyond $F$) when $j$ is odd.
\end{proof}


\begin{observation}[$A$-sewing]\label{obs:Asewing}
In~\cite{LeeMenzel2010}, Lee and Menzel proposed the operation of \emph{$A$-sewing}. Given a flag $\cF=\{F_j\}_{j=1}^k$ of a polytope $P$, it allows to find a point on the facets in $\fF_k$, beyond the facets in $\sew(\cF)\setminus\fF_k$, and beneath the remaining facets. 
In our setting, one can analogously see that the process of $A$-sewing corresponds to a lexicographic extension by $[F_1^+,U_2^-,U_3^+,\dots,U_{k}^{(-1)^{k-1}}]$. 
In the example of Figure~\ref{fig:sewingonflags}, the polytopes $P[c^+,b^-]$ and $P[c^+,b^-,a^+]$ correspond to $A$-sewing through~the flags $\{c\}\subseteq\{c,b\}$ and $\{c\}\subseteq\{c,b\}\subseteq\{c,b,a\}$, respectively.
\end{observation}

\subsection{Sewing onto universal flags}

Shemer's Sewing Construction starts with a neighborly oriented matroid $\cP$ of rank~$\rd$ with $n$ elements and gives a neighborly oriented matroid $\tilde \cP$ of rank~$\rd$ with $n+1$ elements, provided that $\cP$ has a \emph{universal flag}.

\begin{definition}
Let $\cP$ be a uniform acyclic oriented matroid of rank~$\rd$, and let $m=\ffloor{\rd-1}{2}$.
\begin{enumerate}[label={(\roman*)}, leftmargin=*]
 \item A face $F$ of $\cP$ is a \emph{universal face} if the contraction $\cP/F$ is neighborly.
 \item A flag $\cF$ of $\cP$ is a \emph{universal flag} if $\cF=\{F_j\}_{j=1}^m$ where each $F_j$ is a universal face with $2j$ vertices.
\end{enumerate}
\end{definition}

The most basic example of neighborly polytopes with universal flags are cyclic polytopes, (cf. {\cite[Theorem 3.4]{Shemer1982}} and \cite[Theorem 1.1]{CordovilDuchet1990}).

\begin{proposition}[{\cite[Theorem 3.4]{Shemer1982}}]\label{prop:universalflagsofcyclic}
Let $\cyc{2m}{n}$ be a cyclic polytope of dimension $2m$, with vertices $a_1,\dots,a_n$ labeled in cyclic order. Then $\{a_i,a_{i+1}\}$ for $1\leq i<n$ and $\{a_1,a_{n}\}$ are universal edges of $\cyc{2m}{n}$. If moreover $n>2m+2$, then these are all the universal edges of~$\cyc{2m}{n}$.
\end{proposition}

\begin{remark}\label{rmk:universalflagsofcyclic}
It is not hard to prove that, for any universal edge $E$ of $\cyc{2m}{n}$,
\(\cyc{2m}{n}/E\simeq \cyc{2m-2}{n-2}\) 
where the isomorphism is such that the cyclic order is preserved. This observation, combined with Proposition~\ref{prop:universalflagsofcyclic}, provides a recursive method to compute universal flags of $\cyc{2m}{n}$ using universal faces that are the union of a universal edge of $\cyc{2m}{n}$ with a (possibly empty) universal face of $\cyc{2m-2}{n-2}$. 
\end{remark}

With these notions, we are ready to present Shemer's Sewing Theorem.
\begin{theorem}[The Sewing Theorem]\label{thm:shemersewing} \textup{\cite[Theorem 4.6]{Shemer1982}}
Let $P$ be a neighborly $2m$-polytope with a universal flag $\cF=\{F_j\}_{j=1}^m$, where $F_j=\bigcup_{i=1}^{j}\{x_i,y_i\}$.
Let $P[\cF]$ be the polytope obtained by sewing $p$ onto $P$ through~$\cF$.
 Then,
\begin{enumerate}
\item\label{it:newisneigh} $P[\cF]$ is a neighborly polytope with vertices $\verts (P[\cF])=\verts (P)\cup \{p\}$.
\item\label{it:newunifaces} For all $1\leq j \leq m$, $F_{j-1}\cup\{x_j,p\}$ and $F_{j-1}\cup\{y_j,p\}$ are universal faces of $P[\cF]$. If moreover $j$ is even, then $F_j$ is also a universal face of $P[\cF]$.
\end{enumerate} 
\end{theorem}

Combining Remark~\ref{rmk:universalflagsofcyclic} and the Sewing Theorem~\ref{thm:shemersewing}, one can obtain a large family of neighborly polytopes.

\begin{constr}[Sewing: the family~$\cS$]\label{constr:cS}

\hspace*{\fill}
  \begin{itemize}[leftmargin=1cm, rightmargin=.5cm]
  \item Let $P_0:=\cyc{d}{n}$ be an even-dimensional cyclic polytope.
  \item Let $\cF_0$ be a universal flag of $P_0$. It can be found using Remark~\ref{rmk:universalflagsofcyclic}.
  \item For $i=1\dots k$:
  \begin{itemize}[leftmargin=.75cm]
  \item Let $P_i:=P_{i-1}[\cF_{i-1}]$. Then $P_i$ is neighborly by  Theorem~\ref{thm:shemersewing}(\ref{it:newisneigh}).
  \item Theorem~\ref{thm:shemersewing}(\ref{it:newunifaces}) constructs a universal flag $\cF_i$ of $P_i$.
  \end{itemize}
  \item $P:=P_k$ is a neighborly polytope in $\cS$.
 \end{itemize}
\end{constr}

This method generates a family of neighborly polytopes that we call \defn{totally sewn} polytopes
 and denote by \defn{$\cS$}. In contrast to Shemer's original definition of totally sewn polytopes, we do not admit arbitrary universal flags of $P[\cF]$ for sewing, but only those that arise from Theorem~\ref{thm:shemersewing}\eqref{it:newunifaces}.

\subsection{Inseparability: an essential tool}

Before we present our extensions of Shemer's technique, we must introduce an essential (albeit straightforward) tool that will be used extensively in what follows. It is strongly related to the concept of universal edges.

\begin{definition}
Given an oriented matroid~$\cM$ on a ground set $E$, and $\alpha\in\{+1,-1\}$, we say that two elements $p,q\in E$ are \defn{$\alpha$-inseparable} in $\cM$ if \begin{equation}\label{eq:definseparable}X(p)=\alpha X(q)\end{equation}
for each circuit $X\in\ci(\cM)$ with $p,q\in \underline X$. 

In the literature, $(+1)$-inseparable elements are also called \defn{covariant} and $(-1)$-inseparable elements \defn{contravariant} (see~\cite[Section 7.8]{OrientedMatroids1993}). 
\end{definition}

\begin{remark}\label{rmk:insareuni}
 It is not hard to see that if a pair $x, y$ of elements of a neighborly matroid $\cP$ are $(-1)$-inseparable then they form a universal edge of $\cP$. If moreover the rank of $\cP$ is odd, the converse is also true; that is, $x$ and $y$ form a universal edge only if they are $(-1)$-inseparable.
\end{remark}

A first useful property is that inseparability is preserved by duality (with a change of sign).

\begin{lemma}[{\cite[Exercise 7.36]{OrientedMatroids1993}}]\label{lem:insep}
 A pair of elements $p$ and $q$ are $\alpha$-inseparable in $\cM$ if and only if they are $(-\alpha)$-inseparable in $\Gale\cM$.
\end{lemma}

The following lemma about inseparable elements of neighborly and balanced oriented matroids will be also useful later. 
\begin{lemma}\label{lem:balonlycovar}
 All inseparable elements of a a balanced oriented matroid $\cM$ of rank $\rr\geq 2$ with $n$ elements such that $n-\rr-1$ is even must be $(+1)$-inseparable.

 Analogously, all inseparable elements of a neighborly oriented matroid $\cP$ of odd rank $\rd$ with at least $\rd+2$ elements must be $(-1)$-inseparable.
\end{lemma}
\begin{proof}
 Both results are equivalent by duality and Lemma~\ref{lem:insep}. To prove the second claim, observe that if $p$ and $q$ are $\alpha$-inseparable in $\cP$, then they are also $\alpha$-inseparable in $\cP\setminus S$ for any $S$ that contains neither $p$ nor $q$. Hence we can remove elements from $\cP$ until we are left with a neighborly matroid of rank $\rd$ with $\rd+2$ elements. All neighborly matroids of even dimension and corank~$2$ are cyclic $d$-polytopes with $d+3$ vertices (see~\cite[Section~2]{Gale1963}), and those only have $(-1)$-inseparable pairs.
\end{proof}

A final observation is that inseparable elements appear naturally when working with lexicographic extensions.

\begin{lemma}\label{lem:leinseparable}
 If $\cM [p]$ is a lexicographic extension of $\cM$ by $p=[a_1^{\ep_1},\dots,a_k^{\ep_k}]$, then $p$ and~$a_1$ are always $(-\ep_1)$-inseparable. Even more, $p$ and $a_i$ are $(-\ep_i)$-inseparable in $\cM[p]/\{a_1,\dots,a_{i-1}\}$ for $i=1\dots k$, and this property characterizes this single element extension (if $p$ is a loop in $\cM[p]/\{a_1,\dots,a_{k}\}$). 
\end{lemma}

\subsection{Extended Sewing:  flags that contain universal subflags}

We are now almost ready to present our first new construction, a generalized version of the Sewing Theorem for neighborly oriented matroids.
Like \cite[Theorem 2]{Bisztriczky2000}, our Extended Sewing does not depend on the parity of the rank. Moreover, it applies to any flag that contains a universal subflag, as suggested in~\cite[Remark 7.4]{Shemer1982}. 
The analogue of the second part of the Sewing Theorem~\ref{thm:shemersewing} is Proposition~\ref{prop:extnewunifaces}, where we find universal faces of the new neighborly matroid. 

In order to prove that Extended Sewing works, we need the following lemma, which generalizes~\cite[Theorem 3.1]{TrelfordVigh2011}, and the notation $\cF'/F_i=\{F_j'/F_i\}_{j=i+1}^{m}$ where $F_j'/F_i$ is the face of $\cP/F_i$ that represents~$F_j'$. 

\begin{lemma}\label{lem:quotientsofextShemersewing}
Let $\cP$ be a uniform neighborly matroid of rank $\rd$. Let $\cF'=\{F_k'\}_{k=1}^l$ be a flag of $\cP$ that contains a universal subflag $\cF=\{F_j\}_{j=1}^m$, where $m=\ffloor{\rd-1}{2}$ and $F_j=\bigcup_{i=1}^{j}\{x_i,y_i\}$.
 Let $p$ be sewn onto $\cP$ through~$\cF'$.

 If $F_{i-1}\cup \{y_{i}\}$ does not belong to $\cF'$, then
\begin{equation*}
\cP[\cF']/\{F_{i-1},x_{i},p\}\simeq (\cP/F_i)[\cF'/F_i].\end{equation*}
This isomorphism sends $y_{i}$ to the vertex sewn through~$[\cF'/F_i]$, while  the remaining vertices are mapped to their natural counterparts.
\end{lemma}

\begin{proof}
By Proposition~\ref{prop:allquotientsofle}, the contraction $\cP[\cF']/F_{i-1}$ is a lexicographic extension of $\cP/F_{i-1}$ whose signature coincides with that of $[\cF']$ by removing the first $2(i-1)$ elements. Hence $\cP[\cF']/F_{i-1}$ must be one of the extensions
\begin{equation*}
\cP[\cF']/F_{i-1} \in
\left.
\begin{cases} 
\quad\cP/F_{i-1}[x_i^+,y_i^+,x_{i+1}^-,\dots],
\\
\quad\cP/F_{i-1}[x_i^-,y_i^-,x_{i+1}^+,\dots],
\\
\quad\cP/F_{i-1}[x_i^+,y_i^-,x_{i+1}^+,\dots],
\\
\quad\cP/F_{i-1}[x_i^-,y_i^+,x_{i+1}^-,\dots]
\end{cases}\right\}.
\end{equation*}

If $F'_{k-1}$ is the face of $\cF'$ corresponding to $F_{i-1}$, and $U'_{k}=F'_{k}\setminus F'_{k-1}$, then the first two cases are possible when $U'_{k}=\{x_i,y_i\}$, and the last two when $U'_{k}= \{x_i\}$ (the case $U'_{k}= \{y_i\}$ is excluded by hypothesis).
We use Proposition~\ref{prop:allquotientsofle} twice on each of these (contracting successively	 $x_i$ and $p$) to get \(\cP[\cF']/\{F_{i-1},x_{i},p\}\simeq (\cP/\{F_{i-1},x_{i},y_i\})[x_{i+1}^+,\dots]=(\cP/F_i)[\cF'/F_i]\).
\end{proof}

We can now state and prove the Extended Sewing Theorem.

\begin{theorem}[The Extended Sewing Theorem]\label{thm:extshemersewing}
Let $\cP$ be a uniform neighborly oriented matroid of rank~$\rd$ with a flag $\cF'=\{F_k'\}_{k=1}^l$ that contains a universal subflag $\cF=\{F_j\}_{j=1}^m$, where $F_j=\bigcup_{i=1}^{j}\{x_i,y_i\}$ and $m=\ffloor{\rd-1}{2}$.
Let $p$ be sewn onto $\cP$ through~$\cF'$. Then $\cP[\cF']$ is a uniform neighborly matroid of rank $\rd$.
\end{theorem}
\begin{proof}
The proof is by induction on $\rd$. Observe for the base case that all acyclic matroids of rank $1$ or $2$ are neighborly.

Assign the labels to $x_1$ and $y_1$ in such a way that the extension $\cP[\cF']$ is either the lexicographic extension 
\(\cP\left[x_1^+,y_1^+,\dots\right]\) or \(\cP\left[x_1^+,y_1^-,\dots\right]\) (depending on whether $F'_1=\{x_1,y_1\}$ or $F'_1=\{x_1\}$). 

We check that $\cP[\cF']$ is neighborly by checking that $\pGale{\cP[\cF']}$ is balanced, \ie we check that every circuit $X$ of $\cP[\cF']$ is balanced. That is, we want to see that $\ffloor{\rd+1}{2}\leq |X^+|\leq \fceil{\rd+1}{2}$, where $X^+=\set{e\in E}{X(e)=+}$. Let $X\in \ci(\cP)$:
\begin{enumerate}
 \item If $X(p)=0$, then $X$ is balanced because it is also a circuit of $\cP$, and $\cP$ is neighborly.
 \item If $X(p)\neq 0$ and $X({x_1})=0$, we use that $p$ and $x_1$ are $(-1)$-inseparable because of Lemma~\ref{lem:leinseparable}. By Lemma~\ref{lem:circinseparable}, there is a circuit $X'\in \ci(\tilde \cP)$ with $X'({x_1})=X(p)$, $X'(p)=0$ and $X'(e)=X(e)$ for all $e\notin\{x_1,p\}$. Observe that $|X^+|=|X'^+|$. Since $X'(p)=0$, $X'$ is balanced by the previous point, and hence so is~$X$.
 \item If $X(p)\neq 0$ and $X({x_1})\neq 0$ then $X(p)=-X({x_1})$ because $p$ and $x_1$ are $(-1)$-inseparable. Observe that the rest of the values of $X$ correspond to a circuit of $\cP[\cF']/\{p,x_1\}$. If $\cP[\cF']/\{p,x_1\}$ is neighborly, we are done.

By Lemma~\ref{lem:quotientsofextShemersewing}, $\cP[\cF']/\{p,x_1\}\simeq \left(\cP/F_1\right)[\cF'/F_1]$. Since the edge $\{x_1,y_1\}$ was universal, the oriented matroid $\cP/F_1$ (of rank $s-2$) is neighborly, and the flag $\cF'/F_1$  contains the universal flag $\cF/F_1$. Therefore, $\cP[\cF']/\{p,x_1\}$ is neighborly by induction.\qedhere
\end{enumerate}
\end{proof}

\begin{figure}[htpb]
\centering
 \begin{tabular}{ccc}
\includegraphics[width=.25\textwidth]{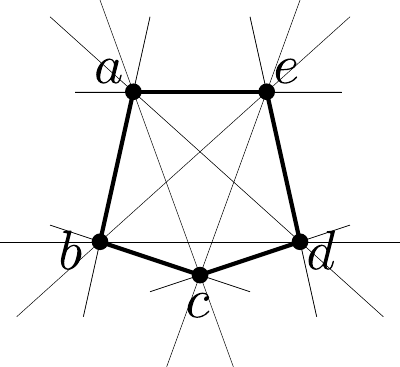}&
\includegraphics[width=.25\textwidth]{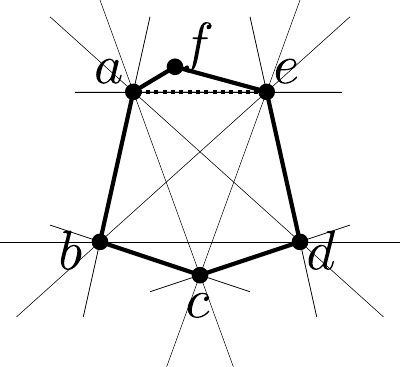}& 
\includegraphics[width=.25\textwidth]{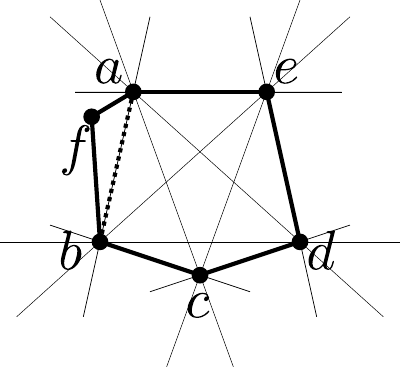}\\
$P$&
$P[a^+,e^+,c^-]$&
$P[a^+,e^-,c^+]$ 
 \end{tabular}
 \caption[Sewing example]{Extended sewing: sewing $f$ onto $\{a,e\}$ (middle), and sewing onto $\{a\}\subset \{a,e\}$ (right). In the first case, $\{a,f\}$ and $\{e,f\}$ become universal faces, while $\{a,e\}$ is not a universal face any more. In the second case, $\{a,f\}$ and $\{a,e\}$ are universal faces, while $\{e,f\}$ is not.  }
 \label{fig:sewing}
\end{figure}

One way to understand this technique is the following. By construction, $p$ is beneath every facet of $\cP$ that does not contain $x_1$. Therefore, every subset $S$ of $\ffloor{s-1}{2}$ elements of $\cP$ that does not contain $x_1$ must still be a face of $\cP[p]$. Hence, to prove the neighborliness of $\cP[p]$, it is enough to study those subsets that contain $x_1$ or~$p$. For those, we use Lemma~\ref{lem:quotientsofextShemersewing}. If $\cF'$ is chosen to contain a universal subflag, then the contraction of $\{x_1,p\}$ is also an Extended Sewing of a neighborly matroid; and thus, neighborly by induction.

A first application of the Extended Sewing Theorem is the construction of cyclic polytopes.

\begin{proposition}[{\cite[Theorem 5.1]{LeeMenzel2010}}]\label{prop:cyclicareextendedsewn}
Let $\cP$ be the oriented matroid of a cyclic polytope $\cyc{d}{n}$ with elements $a_1,\dots,a_n$ labeled in cyclic order, and let~$\cF$ be the flag $\cF=\{a_n\}\subset\{a_{n-1},a_n\}\subset \dots \subset\{a_{n-d+1},\dots,a_n\}$. Then $\cP[\cF]$ is the oriented matroid of the cyclic polytope $\cyc{d}{n+1}$.
\end{proposition}

\subsection{Universal faces created by Extended Sewing}

We can tell many universal faces of the neighborly oriented matroids constructed using the Extended Sewing Theorem~\ref{thm:extshemersewing} thanks to Proposition~\ref{prop:extnewunifaces}, the analogue of the second part of the Sewing Theorem~\ref{thm:shemersewing}. It provides a simple way to compute universal flags of sewn matroids that is explained in Remark~\ref{rmk:extnewuniflags}.

These faces are best described using the following notation for flags that contain a fixed universal subflag.

\begin{definition}
Let $\cP$ be a neighborly matroid of rank~$\rd=2m+1$ and let $\cF'=\{F_k'\}_{k=1}^l$ be a flag of $\cP$ that contains the universal subflag $\cF=\{F_j\}_{j=1}^m$, where $F_j=\bigcup_{i=1}^{j} \{x_i,y_i\}$. Observe that for each $1\leq i\leq j$, $F_{i-1}\cup\{x_i\}$ and $F_{i-1}\cup\{y_i\}$ cannot both belong to~$\cF'$.
We say that $F_{i}\in \cF$ is \defn{$x_i$-split} (resp. \defn{$y_i$-split}) in $\cF'$ if $F_{i-1}\cup\{x_i\}$ (resp. $F_{i-1}\cup\{y_i\}$) belongs to~$\cF'$, and \defn{non-split} if neither $F_{i-1}\cup\{x_i\}$ nor $F_{i-1}\cup\{y_i\}$ belong to $\cF'$.
Moreover, we say that $F_i$ is \defn{even} in $\cF'$ if the number of non-split faces~$F_j$ with $j\leq i$ is even, $F_i$ is \defn{odd} otherwise.
\end{definition}

For example, if $l=2$ and $\cF=(F_1:=	\{x_1,y_1\})\subset (F_2:=\{x_1,y_1,x_2,y_2\})$ is a universal flag, then $F_1$ is $x_1$-split and $F_2$ is non-split in the flag $\cF'= \{x_1\}\subset\{x_1,y_1\}\subset\{x_1,y_1,x_2,y_2\}$. Moreover, $F_1$ is even in $\cF'$ whereas $F_2$ is odd. In comparison, in the flag $\cF''=\{x_1,y_1\}\subset\{x_1,y_1,y_2\}\subset\{x_1,y_1,x_2,y_2\}$, $F_1$ is non-split and $F_2$ is $y_2$-split; and both $F_1$ and $F_2$ are odd.

\begin{remark}
 Theorem~\ref{thm:extshemersewing} not only generalizes the Sewing Theorem (when no face is split), but also includes Barnette's facet-splitting technique~\cite[Theorem 3]{Barnette1981}, which corresponds to the case where all faces of the universal flag are split.
\end{remark}

\begin{proposition}\label{prop:extnewunifaces}
Let $\cP$ be a uniform neighborly oriented matroid of rank~$\rd$ with a flag $\cF'=\{F_k'\}_{k=1}^l$ that contains a universal subflag $\cF=\{F_j\}_{j=1}^m$, where $F_j=\{x_i,y_i\}_{i=1}^{j}$ and $m=\ffloor{\rd-1}{2}$. 
Let $p$ be sewn onto $\cP$ through~$\cF'$.
Then the following are universal faces of $\cP[\cF']$:
\begin{enumerate}
 \item\label{it:extnewuf1} $F_i$, where $1\leq i \leq m$, if $F_i$ is even.
 \item\label{it:extnewuf2} $(F_{j}\setminus{x_i})\cup p$, where $1\leq i\leq j\leq m$, if
      \begin{enumerate}[label={(\roman*)}, leftmargin=*]
	\item\label{it:extnewuf21} $F_i$ is not split and $F_j/F_i$ is even in $\cF'/F_i$, or
	\item\label{it:extnewuf22} $F_i$ is $x_i$-split and $F_j/F_i$ is odd in $\cF'/F_i$, or
	\item\label{it:extnewuf23} $F_i$ is $y_i$-split and $F_j/F_i$ is even in $\cF'/F_i$.
      \end{enumerate}
 \item\label{it:extnewuf3} $(F_{j}\setminus{y_i})\cup p$, where $1\leq i\leq j\leq m$, if
      \begin{enumerate}[label={(\roman*)}, leftmargin=*]
	\item\label{it:extnewuf31} $F_i$ is not split and $F_j/F_i$ is even in $\cF'/F_i$, or
	\item\label{it:extnewuf32} $F_i$ is $x_i$-split and $F_j/F_i$ is even in $\cF'/F_i$, or
	\item\label{it:extnewuf33} $F_i$ is $y_i$-split and $F_j/F_i$ is odd in $\cF'/F_i$.
      \end{enumerate}
\end{enumerate}

\end{proposition}
\begin{proof}
Without loss of generality, we will assume that all split faces are $x_i$-split.
The proof relies on applying, case by case, Proposition~\ref{prop:allquotientsofle} to reduce the contraction to a lexicographic extension that we know to be neighborly because of Theorem~\ref{thm:extshemersewing}.

By Definition~\ref{def:sewing}, there are some elements $a,b$ and some $\ep=\pm$ such that
\begin{equation*}
\cP[\cF']
=
\begin{cases}
\cP[\dots,a^\ep,x_i^{-\ep},y_i^{\ep},b^{-\ep},\dots]& \text{ if $F_i$ is $x_i$-split,}
\\
\cP[\dots,a^\ep,x_i^{-\ep},y_i^{-\ep},b^{\ep},\dots]&\text{ if it is not split.}
\end{cases}
\end{equation*}
Therefore, the sign of $x_i$ in $[\cF']$ is $+$ if and only if $F_{i-1}$ is even. In particular, if $F_i$ is even, then $\cP[\cF']/F_i \simeq (\cP/F_i)[\cF'/F_i]$ and $\cF'/F_i$ is a universal flag of $\cP/F_i$, which is neighborly since $F_i$ is a universal face. This proves point~\ref{it:extnewuf1}.

Moreover, independently of whether $F_i$ is even or odd, 
\begin{equation*}
\cP[\cF']/(F_{i-1}\cup \{p\})\simeq  
\begin{cases}
 \cP/(F_{i-1}\cup \{x_i\})[y_i^{+},x_{i+1}^{-},\dots]& \text{ if $F_i$ is split,}
\\
 \cP/(F_{i-1}\cup \{x_i\})[y_i^{-},x_{i+1}^{+},\dots]& \text{ if it is not.}
\end{cases}
\end{equation*}
Hence, $\cP[\cF']/(F_{i-1}\cup \{x_i,p\})\simeq (\cP/F_i)[\cF'/F_i]$ always. If moreover $F_i$ is not split then $\cP[\cF']/(F_{i-1}\cup \{y_i,p\})\simeq (\cP/F_i)[\cF'/F_i]$. Therefore, in these cases the problem is reduced to finding universal faces of $(\cP/F_i)[\cF'/F_i]$. But we already know that $F_j/F_i$ is a universal face of $(\cP/F_i)[\cF'/F_i]$ when it is even. This proves points~\ref{it:extnewuf2}\ref{it:extnewuf21}, \ref{it:extnewuf2}\ref{it:extnewuf23}, \ref{it:extnewuf3}\ref{it:extnewuf31} and \ref{it:extnewuf3}\ref{it:extnewuf32}.

If $F_i$ is split, then $\tilde \cP/(F_{i-1}\cup \{y_i,p\})\simeq (\cP/F_i)[-\cF'/F_i]$, where $[-\cF'/F_i]$ means the extension by $[\cF'/F_i]$ with the signs reversed. Using the previous observation, we obtain that $(\cP/F_i[-\cF'/F_i])/(F_j/F_i)\simeq (\cP/F_j)[\cF'/F_j]$ when $F_j/F_i$ is odd, and this proves the remaining points~\ref{it:extnewuf2}\ref{it:extnewuf22} and~\ref{it:extnewuf3}\ref{it:extnewuf33}.
\end{proof}

\begin{remark}\label{rmk:extnewuniflags}
In particular, Proposition~\ref{prop:extnewunifaces} provides a simple way to tell universal flags of~$\cP[\cF']$. We start with universal edges:

\begin{itemize}
\item If $F_1$~is not split then $\{x_1,p\}$ and $\{y_1,p\}$ are universal edges of $\cP[\cF']$; 
\item if $F_1$~is $x_1$-split, then $\{x_1,p\}$ and $\{x_1,y_1\}$ are universal edges of $\cP[\cF']$; 
\item finally, if $F_1$~is $y_1$-split, then $\{y_1,p\}$ and $\{x_1,y_1\}$ are universal edges of $\cP[\cF']$. 
\end{itemize}
The contraction of any of these universal edges is isomorphic to $(\cP/F_1)[\cF'/F_1]$, and we can inductively build a universal flag of $\cP[\cF']$.
\end{remark}

The example in Figure~\ref{fig:sewing} can give some intuition on why do these universal edges appear. The next example explores higher dimensional universal faces.

\begin{example}\label{ex:universalfaces}
 Let $\cM$ be a neighborly oriented matroid of rank~$5$ with a universal flag $\cF=F_1\subset F_2$, where $F_1=\{a,b\}$ and $F_2=\{a,b,c,d\}$. Consider the lexicographic extensions by the elements
\begin{align*}
p_1&=[a^+,b^+,c^-,d^-,e^+],\\p_2&=[a^+,b^-,c^+,d^+,e^-],\\p_3&=[a^+,b^+,c^-,d^+,e^-],\text{ and }\\p_4&=[a^+,b^-,c^+,d^-,e^+],
\end{align*}
where $e$ is any element of $\cM$. For $i=1,2,3,4$, each $p_i$ gives  rise to the oriented matroid $\cM_i=\cM[p_i]$, which corresponds to sewing through the flag~$\cF_i$, with
\begin{align*}
\cF_1&=\{a,b\}\subset\{a,b,c,d\},\\
\cF_2&=\{a\}\subset\{a,b\}\subset\{a,b,c,d\},\\
\cF_3&=\{a,b\}\subset\{a,b,c\}\subset\{a,b,c,d\},\text{ and }\\
\cF_4&=\{a\}\subset\{a,b\}\subset\{a,b,c\}\subset\{a,b,c,d\}.
\end{align*}
Observe that $F_1$ is split in $\cF_2$ and $\cF_4$, while $F_2$ is split in $\cF_3$ and $\cF_4$. Moreover, $F_1$ is even in $\cF_2$ and $\cF_4$, and $F_2$ is even in $\cF_1$ and $\cF_4$.
Table~\ref{tb:exunifaces} shows for which~$\cM_i$ each of the following sets of vertices is a universal face.

\begin{table}[htpb]
  \caption{Universal faces in Example~\ref{ex:universalfaces}}	\label{tb:exunifaces}
   \renewcommand{\arraystretch}{1.2}
\centering
\newcolumntype{C}[1]{>{\centering\let\newline\\\arraybackslash\hspace{0pt}}m{#1}}
\begin{tabular}{ccC{.9cm}C{.9cm}C{.9cm}C{.9cm}C{.9cm}C{.9cm}C{.9cm}C{.9cm}}
\hline
	&&$ab$&$p_ib$&$ap_i$&$abcd$&$p_ibcd$&$ap_icd$&$abp_id$&$abcp_i$\\\hline
$\cM_1$&&\cross&\tick&\tick&\tick&\cross&\cross&\tick&\tick\\
$\cM_2$&&\tick&\cross&\tick&\cross&\tick&\cross&\tick&\tick\\
$\cM_3$&&\cross&\tick&\tick&\cross&\tick&\tick&\cross&\tick\\
$\cM_4$&&\tick&\cross&\tick&\tick&\cross&\tick&\cross&\tick\\ 
\hline
\end{tabular}
\end{table}

\end{example}

\subsection{Extended Sewing and Omitting}

Just like in the construction of the family {$\cS$}, we can combine the Extended Sewing Theorem~\ref{thm:extshemersewing}
and Proposition~\ref{prop:extnewunifaces} to obtain a large family~\defn{$\cE$} of neighborly polytopes that contains $\cS$. In fact, since cyclic polytopes belong to~$\cE$ by Proposition~\ref{prop:cyclicareextendedsewn} it suffices to start sewing on a simplex.

\begin{constr}[Extended Sewing: the family $\cE$]\label{constr:cE}

\hspace*{\fill}
  \begin{itemize}[leftmargin=1cm, rightmargin=.5cm]
  \item Let $P_0:=\Delta_d$ be a $d$-dimensional simplex.
  \item Let $\cF_0'$ be a flag of $P_0$ that contains a universal subflag $\cF_0$. $\cF_0$ is built using the fact that all edges of a simplex are universal.
  \item For $i=1\dots k$:
  \begin{itemize}[leftmargin=.75cm]
  \item Let $P_i:=P_{i-1}[\cF_{i-1}']$, which is neighborly by Theorem~\ref{thm:extshemersewing}.
  \item Use Remark~\ref{rmk:extnewuniflags} to find a universal flag $\cF_i$ of $P_i$.
  \item Let $\cF_i'$ be any flag of $P_i$ that contains $\cF_i$ as a subflag.
  \end{itemize}
  \item $P:=P_k$ is a neighborly polytope in $\cE$.
 \end{itemize}
\end{constr}

Moreover, since subpolytopes (convex hulls of subsets of vertices) of neighborly polytopes are neighborly, any polytope obtained from a member of~$\cE$ by omitting some vertices is also neighborly. The polytopes that can be obtained in this way via sewing and omitting form a family that we denote \defn{$\cO$}.

\begin{constr}[Extended Sewing and Omitting: the family~$\cO$]\label{constr:cO}

\hspace*{\fill}
  \begin{itemize}[leftmargin=1cm, rightmargin=.5cm]
  \item Let $Q\in \cE$ be a neighborly polytope constructed using Extended Sewing.
  \item Let $S\subseteq\verts(Q)$ be a subset of vertices of $Q$.
  \item $P:=\conv(S)$ is a neighborly polytope in $\cO$.
 \end{itemize}
\end{constr}

\subsection{Optimality}

We finish this section by showing that for matroids of odd rank, the flags of the Extended Sewing Theorem~\ref{thm:extshemersewing} are the only ones that yield neighborly polytopes. Therefore, in this sense the Sewing Construction cannot be further improved.

\begin{proposition}\label{prop:uniqueflags}
Let $\cP$ be a uniform neighborly oriented matroid of odd rank $\rd\geq 3$ with more than $\rd+1$ elements.
Then $\cP[\cF]$ is neighborly if and only if $\cF$ contains a universal subflag.
\end{proposition}
\begin{proof}
By Theorem~\ref{thm:extshemersewing}, this condition is sufficient.
To find necessary conditions, we use that $\cP[\cF]$ is neighborly if and only if every circuit of $\cP[\cF]$ is balanced. 

The proof is by induction on $\rd$. For the base case $\rd=3$ just observe that neighborly matroids of rank $3$ are polygons, and the only 
flags that yield a polygon with one extra vertex are of the form $\{x\}\subset \{x,y\}$ or just $\{x,y\}$, where $\{x,y\}$ is an edge of the polygon.

Assume then that $\rd>3$. By definition, $\cP[\cF]$ is the lexicographic extension $\cP[p]$, with $p$ sewn through~$\cF$. Therefore, \(p=[a_1^{+},\;a_2^{\ep_2},\dots,\;a_{\rd}^{\ep_\rd}]\). 
Let $X\in\ci(\cP[\cF])$ be a circuit with $\{p,a_1\}\subset\underline X$.
Since $p$ and $a_1$ are $(-1)$-inseparable by Lemma~\ref{lem:leinseparable}, $X(p)=-X({a_1})$. Hence, if $X$ is balanced, so is $X\setminus \{p,a_1\}$. Now $X\setminus \{p,a_1\}$ is a circuit of $\cP[\cF]/\{p,a_1\}$, and  all  circuits of $\cP[\cF]/\{p,a_1\}$ arise this way. Hence $\cP[\cF]/\{p,a_1\}$ is neighborly.

By Proposition~\ref{prop:allquotientsofle}, \begin{equation*}\label{eq:Pquopa1}\cP[\cF]/\{p,a_1\}\simeq \cP/\{a_1,a_2\}[a_3^{-\ep_2\ep_3},\;a_4^{-\ep_2\ep_4},\;,\dots,\;a_{\rd}^{-\ep_2\ep_\rd}],\end{equation*}
where the second extension is by $a_2$. Hence, by Lemma~\ref{lem:leinseparable}, $a_2$ and $a_3$ are $(\ep_2\ep_3)$-inseparable in $\cP[\cF]/\{p,a_1\}$, which is a neighborly matroid of odd rank and corank at least~$2$. By Lemma~\ref{lem:balonlycovar}, $\ep_2\ep_3=-$.

In particular, either $(\ep_2, \ep_3)=(+,-)$, or $(\ep_2, \ep_3)=(-,+)$. The first option implies that $F_1=\{a_1,a_2\}$, and the second one that $F_1=\{a_1\}$ and $F_2=\{a_1,a_2\}$. 

Since $(\cP[\cF]/\{p,a_1\})\setminus a_2\simeq\cP/\{a_1,a_2\}$ by Lemma~\ref{lem:contractdeletele}, if $\cP[\cF]/\{p,a_1\}$ is neighborly, then $\cP/\{a_1,a_2\}$ must be neighborly and hence $F:=\{a_1,a_2\}$ must be a universal edge of $\cP$ that belongs to $\cF$. 

Finally, observe that $\cP[\cF]/F=(\cP/F)[\cF/F]$ is a matroid of rank $\rd -2$. By induction, $\cF/F$ contains a universal subflag. The union of $F$ with each universal face in $\cF/F$ is a universal face of $\cP$ in $\cF$, which finishes the proof.
\end{proof}

\section{The Gale Sewing Construction}\label{sec:thethm}

In this section, we present a different method to construct neighborly matroids. It is also based on lexicographic extensions, but
 works in the dual, that is, it extends balanced matroids to new balanced matroids.
The key ingredient is the Double Extension Theorem~\ref{thm:thethm}, which shows how to perform double element extensions that preserve balancedness. Before proving it, we need a small lemma.
\begin{lemma}\label{lem:quotientsofGalesewnareGalesewn}
Let $\cM$ be a uniform oriented matroid of rank $r$, let $a_1\dots a_r$ be elements of $\cM$ and $\ep_1,\dots,\ep_r$ be signs. If  $p$, $q$, $p'$ and $q'$ are defined as
\begin{align*}
p&=[a_1^{\ep_1},a_2^{\ep_2},\dots,a_r^{\ep_r}],&q&=[p^-,a_1^{-},\dots,a_{r-1}^{-}];\\
p'&=[a_2^{-\ep_1\ep_2},\dots,a_r^{-\ep_1\ep_r}],& q'&=[p'^{-},\dots,a_{r-1}^{-}],
\end{align*}
then \[
\left(\cM[p][q]\right)/q\ \simeq\ \left(\cM/a_1\right)[p'][q'].
\]
\end{lemma}
\begin{proof}
Repeatedly applying Proposition~\ref{prop:allquotientsofle}:
\begin{align*}
\left(\cM[p][q]\right)/q&= \big(\cM\underbrace{[a_1^{\ep_1},a_2^{\ep_2},\dots,a_r^{\ep_r}]}_{p}\underbrace{[p^-,a_1^{-},\dots,a_{r-1}^{-}]}_{q}\big)/q\\
&\stackrel{\varphi}{\simeq} \big(\cM\underbrace{[a_1^{\ep_1},a_2^{\ep_2},\dots,a_r^{\ep_r}]}_{p}/p\big)\underbrace{[a_1^{-},\dots,a_{r-1}^{-}]}_{\varphi(p)=q'}\\
&\stackrel{\psi}{\simeq} \big(\cM/a_1\big)\underbrace{[a_2^{-\ep_1\ep_2},\dots,a_r^{-\ep_1\ep_r}]}_{\psi(a_1)=p'}\underbrace{[\psi(a_1)^{-},\dots,a_{r-1}^{-}]}_{q'}.\qedhere
\end{align*}
\end{proof}

\begin{theorem}[Double Extension Theorem]\label{thm:thethm}
Let $\cM$ be a uniform balanced oriented matroid of rank $r$. For any sequence $a_1\dots a_r$ of elements of~$\cM$ and any sequence $\ep_1,\dots,\ep_r$ of signs,
consider the lexicographic extension 
\begin{itemize}
\item $\cM[p]$ of $\cM$ by $p=[a_1^{\ep_1},a_2^{\ep_2},\dots,a_r^{\ep_r}]$, and 
\item $\cM[p][q]$ of $\cM[p]$ by $q=[p^-,a_1^{-},\dots,a_{r-1}^{-}]$; 
\end{itemize}
then the oriented matroid $\cM[p][q]$ is balanced.
\end{theorem}
\begin{proof}
The proof is by induction on $r$ (it is trivial for $r=0$). For $r\geq1$ we check that every cocircuit $\tilde C$ of $\cM[p][q]$ is balanced. That is, for each cocircuit $\tilde C\in\co(\cM[p][q])$, we prove that $\ffloor{n-r+1}{2}\leq |\tilde C^+|\leq\fceil{n-r+1}{2}$, where $n$ is the number of elements of $\cM[p][q]$ and $\tilde C^+=\set{e\in E}{C(e)=+}$. 

If $\tilde C(p)\neq 0$ and $\tilde C(q)\neq 0$ then, by the definition of lexicographic extension, there is a cocircuit $C$ of~$\cM$ such that $\restriction{\tilde C}{\cM}=C$ and $\tilde C(p)=-\tilde C(q)$. Hence $|\tilde C^+|=|C^+|+1$, and it is balanced because $C$ is a balanced circuit of $\cM$ (observe that $\cM$ has $n-2$ elements).

The cocircuits $\tilde C$ with $\tilde C(p)=0$ correspond to cocircuits of $(\cM[p][q])/p$, and those with $\tilde C(q)=0$ correspond to cocircuits of $(\cM[p][q])/q$. Therefore, it is enough to prove that $(\cM[p][q])/p$ and $(\cM[p][q])/p$ are balanced.
By Proposition~\ref{prop:allquotientsofle} and Lemma~\ref{lem:quotientsofGalesewnareGalesewn}
\[(\cM[p][q])/p\simeq (\cM[p][q])/q\simeq(\cM/a_1)\underbrace{[a_2^{-\ep_1\ep_2},\dots,a_r^{-\ep_1\ep_r}]}_{p'}[{p'}^-,a_2^-,\dots,a_{r-1}^-],\]
which is a double extension of the balanced matroid $\cM/a_1$ of rank $r-1$, and therefore a balanced matroid by induction.
\end{proof}
\begin{figure}[htpb]
\centering
\includegraphics[width=.5\textwidth]{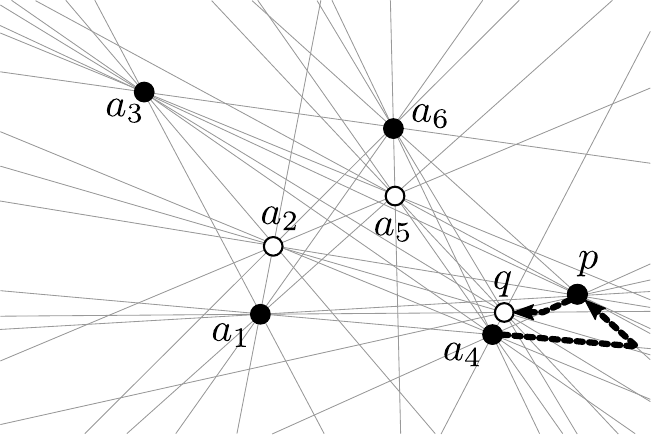}
\caption{The double lexicographic extension of an affine Gale balanced diagram by $p=[a_4^+,a_1^-,a_6^+]$ and $q=[p^-,a_4^-,a_1^-]$, which is also balanced.}\label{fig:galesewingexample}
\end{figure}

If $V$ is a balanced vector configuration, the proof that $V[p][q]$, its lexicographic extension by $p=[a_1^{\ep_1},\dots,a_r^{\ep_r}]$ and $q=[p^-,\dots,a_{r-1}^{-}]$, is also balanced is very easy to understand. Every hyperplane $H$ spanned by a subset of $V$ defines a cocircuit of $V[p][q]$.
The signature of the extension by $q$ implies that if $p\in H^\pm$ then $q\in H^\mp$, and hence $q$ balances the discrepancy created by $p$ on this hyperplane.
The other hyperplanes are checked inductively. Indeed, for a hyperplane $H$ that contains~$p$ but neither~$a_1$ nor~$q$, the fact that $p$ and $a_1$ are inseparable implies that except for~$a_1$, $H$ looks like a hyperplane spanned by $V$ containing~$a_1$.
Hence $q$ must balance the discrepancy created by $a_1$.
For hyperplanes that go through $p$ and $a_1$ but neither $a_2$ nor $q$, $q$ balances the discrepancy created by~$a_2$; and so on.

Figure~\ref{fig:galesewingexample} displays an example of such a double extension on an affine Gale diagram. 
The reader is invited to follow this justification in the picture (for example, by comparing the hyperplanes spanned by $\{a_4,a_i\}$ with the hyperplanes spanned by $\{p,a_i\}$) and to check how all cocircuits in the diagram are balanced.

\begin{corollary}\label{cor:primalGaleSewing}
For any neighborly matroid $\cP$ of rank $\rd$ and $n$ elements there is a neighborly matroid $\tilde \cP$ of rank $\rd+2$ with $n+2$ elements that has an edge $\{x,y\}$ such that $\tilde\cP/\{x,y\}=\cP$.
\end{corollary}

\begin{remark}
In fact, the proof of Theorem~\ref{thm:thethm} shows a stronger result: For a uniform, not necessarily balanced oriented matroid $\cM$ on which this pair of extensions is performed, the maximal difference between the number of positive and negative elements of a cocircuit (its discrepancy) does not increase.
\end{remark}

This provides the following method to construct balanced matroids (and hence, by duality, to construct neighborly matroids). 
\begin{constr}[Gale Sewing: the family~$\cG$]\label{constr:cG}

\hspace*{\fill}
  \begin{itemize}[leftmargin=1cm, rightmargin=.5cm]
  \item Let $\cM_0$ be the minimal totally cyclic oriented matroid, realized by $\{e_1, \dots, e_r, -\sum_{i=1}^r e_i\}$, where $\left\{e_i\right\}_{1\leq i\leq r}$ is the standard basis.
  \item For $k=1\dots m$:
  \begin{itemize}[leftmargin=.75cm]
  \item Choose different elements $a_{k1},\dots, a_{kr}$ of $\cM_{k-1}$ and choose $\ep_{kj}\in\{+, -\}$ for $j=1\dots r$.
  \item Let $ p_k:=[ a_{k1}^{\ep_{k1}},\dots, a_{kr}^{\ep_{kr}}]$ and $ q_k:=[ p_k^-, a_{k1}^{-},\dots, a_{k(r-1)}^{-}]$.
  \item $\cM_k:=\cM_{k-1}[ p_k][ q_k]$ is balanced because of Theorem~\ref{thm:thethm} and realizable because of Lemma~\ref{lem:realizablele}.
  \end{itemize}
  \item $\cM:=\cM_k$ is a realizable balanced oriented matroid.
  \item $\cP:=\Gale\cM$ is a realizable neighborly oriented matroid. 
  \item Any realization $P$ of $\cP$ is a neighborly polytope in $\cG$.
 \end{itemize}
\end{constr}
We call the double extension of Theorem~\ref{thm:thethm} \defn{Gale Sewing}, and we denote by \defn{$\cG$} the family of combinatorial types of polytopes whose dual is constructed by repeatedly Gale Sewing from $\{e_1,\dots,e_r,-\sum_{i=1}^r e_i\}$. If $P\in\cG$, we will say that $P$ is \defn{Gale sewn}.

\begin{remark}\label{rmk:uflag}
 With the notation of Construction~\ref{constr:cG}, observe that the set $F_j:=\bigcup_{i=0}^{j-1} \{p_{m-i},q_{m-i}\}$ is always a universal face of $\cP$ (that is, $\cP/F_j$ is neighborly), since $\cM\setminus F_j$ is balanced. In particular, $\cF:=\{F_i\}_{i=1}^{m}$ is a universal flag of $\cP$.
\end{remark}

\begin{remark}\label{rmk:odddim}
 In the formulation above, Construction~\ref{constr:cG} only allows for constructing even dimensional neighborly polytopes. To construct odd dimensional polytopes it is enough to do one arbitrary single element extension to one $\cM_i$ for some $0\leq i \leq m$. It is straightforward to check that the matroid obtained after such an extension is balanced (and hence also all its double extensions).
\end{remark}

Cyclic polytopes are a first example of polytopes in $\cG$. The following proposition shows that $\pGale{\cyc{d+1}{n+1}}\simeq \pGale{ \cyc{d}{n}}[a_{n}^-,\dots,a_{d}^-]$. Therefore, every even dimensional cyclic polytope $\cyc{d}{n}$ can be obtained from $\cyc{d-2}{n-2}$ with a double extension in the sense of Theorem~\ref{thm:thethm}:
\[\pGale{ \cyc{d}{n}}\simeq \pGale{ \cyc{d-2}{n-2}}[a_{n-2}^-,\dots,a_{d-2}^-][a_{n-1}^-,\dots,a_{d-1}^-].\]
This implies that cyclic polytopes are in $\cG$ because the base case of Construction~\ref{constr:cG} corresponds to $0$-dimensional cyclic polytopes. Observe that this proposition also explains how to construct odd dimensional cyclic polytopes $\cyc{d}{n}$: their duals correspond to a single lexicographic extension of $\pGale{\cyc{d-1}{n-1}}$.

\begin{proposition}\label{prop:cyclicaregalesewn}
 Let $\cM$ be the dual of the alternating matroid of the cyclic polytope $\cyc{d}{n}$, and let $a_1,a_2,\dots, a_n$ be its elements labeled in cyclic order. Then the dual oriented matroid of $\cyc{d+1}{n+1}$ is $\cM[a_{n+1}]$, the single element extension of $\cM$ by $a_{n+1}=[a_{n}^-,a_{n-1}^-,\dots,a_{d}^-]$.
\end{proposition}
\begin{proof}
We use the following characterization of the circuits of the alternating matroid of rank~$r$
(cf. \cite[Section 9.4]{OrientedMatroids1993}):
the circuits $X$ and $Y$ supported by the $r+1$ elements $x_1<x_2<\dots<x_{r+1}$ (sorted in cyclic order) are those such that $X(x_i)=(-1)^i$ and $Y(x_i)=(-1)^{i+1}$.

If $C$ is a cocircuit of $\cM[a_{n+1}]$ (hence a circuit of its dual) such that $C(a_{n+1})\neq 0$, the signature of the lexicographic extension implies that $C(a_{n+1})$ is opposite to the sign of the largest non-zero element. And thus, by the characterization above, $\cM[a_{n+1}]$ is dual to $\cyc{d+1}{n+1}$. 
\end{proof}

Finally, the following proposition shows that subpolytopes (convex hulls of subsets of vertices) of Gale sewn polytopes are also Gale sewn polytopes. Its proof, which is easy using Proposition~\ref{prop:allquotientsofle} and Lemma~\ref{lem:quotientsofGalesewnareGalesewn}, can be found in Appendix~\ref{sec:appendix}.

\begin{proposition}\label{prop:allquotientsofGalesewnareGalesewn}
 If $P$ is a neighborly polytope in $\cG$, and $a$ is a vertex of~$P$, then $Q=\conv(\verts(P)\setminus a)$ is also a neighborly polytope in $\cG$.
\end{proposition}

\subsection{Combinatorial description of the polytopes in $\cG$}

Let $P$ be a simplicial polytope that defines an acyclic uniform oriented matroid $\cP$, and let $\cM=\Gale\cP$ be its dual matroid. The essence of Gale Sewing is to construct a new polytope ${\tilde P}$ whose matroid $\tilde \cP$ is dual to $\tilde\cM=\cM[p]$, a lexicographic extension of $\cM$ by $p=[a_1^{\ep_1},a_2^{\ep_2},\dots,a_k^{\ep_k}]$. In this section we will see that the combinatorics of ${\tilde P}$ are described by \defn{lexicographic triangulations} of $P$.

Let $A=\{a_1,\dots,a_n\}$ be the set of vertices of $P\subset \RR^d$. Let $M$ be the $d\times n$ matrix whose columns list the coordinates of the $a_i$'s:
\[
\renewcommand{\arraystretch}{1.5}
       M := 
       \left[
\begin{matrix}
& \kern.4em\vrule height 2ex\kern.2em & \kern.2em\vrule height 2ex\kern.2em& & \vrule height 2ex\kern.6em &\\
&\kern.2em\,a_1\,&\,a_2\,&\dots&\,a_n\,&\\
& \kern.4em\vrule depth 0ex\kern.2em & \kern.2em\vrule depth 0ex\kern.2em& & \vrule depth 0ex\kern.6em &\\
\end{matrix}
 	\right].	
\]

Then there is some small $\delta>0$ such that the point configuration ${\tilde {A}}$ defined by the columns of the following $(d+1)\times(n+1)$ matrix $\tilde M$ is a realization of the set of vertices of~${\tilde {P}}$:
\[
\renewcommand{\arraystretch}{1.5}
      \tilde M := 
\kbordermatrix{
&\tilde{a}_1&\tilde{a}_2&\tilde{a}_3&\dots&\tilde{a}_k&\tilde{a}_{k+1}&\dots &\tilde{a}_n& \omit\vrule &p \\
&\kern.4em\vrule height 2ex\kern.4em & \kern.4em\vrule height 2ex\kern.4em & \kern.4em\vrule height 2ex\kern.4em& & \vrule height 2ex\kern.6em & \vrule height 2ex\kern.6em & & \kern.4em\vrule height 2ex\kern.4em & \omit\vrule & \kern.4em\vrule height 2ex\kern.4em \\
&\kern.4ema_1&a_2&a_3&\dots&a_k&a_{k+1}&\dots &a_n& \omit\vrule & \veczero \\
&\kern.4em\vrule depth 1ex\kern.4em & \kern.4em\vrule depth 1ex\kern.4em & \kern.4em\vrule depth 1ex\kern.4em& & \vrule depth 1ex\kern.6em & \vrule depth 1ex\kern.6em & & \kern.4em\vrule depth 1ex\kern.4em& \omit\vrule & \kern.4em\vrule depth 1ex\kern.4em \\
\cline{2-11}
&-\ep_1&-\ep_2 \delta &-\ep_3\delta^2& \dots &-\ep_k\delta^{k-1}&0\kern.4em&\dots&0& \omit\vrule &\kern.2em 1\kern.2em
}.
\]

Geometrically, each point $a_i\in A\subset \RR ^d$ is lifted to a point $\tilde {a}_i\in\tilde{A}\subset \RR^{d+1}$ with a height that depends on the signature of the lexicographic extension. Namely, $\tilde{a}_i=\binom{a_i}{ -\ep_i\delta^{i-1}}$ for $i\leq k$ and $\tilde{a}_i=\binom{a_i}{0}$ otherwise. Moreover, $p$ is added to $\tilde{A}$ with coordinates $\binom{\veczero}{1}$.
The vertex figure of $p$ in~${\tilde {P}}$ is combinatorially equivalent to $P$. That is, the faces of ${\tilde{P}}$ that contain $p$ are isomorphic to pyramids over faces of $P$. 
On the other hand, the faces of~${\tilde {P}}$ that do not contain $p$ correspond to faces of a regular subdivision of $P$: the \defn{lexicographic subdivision} of $P$ on $[a_1^{-\ep_1},a_2^{-\ep_2},\dots,a_k^{-\ep_k}]$. When $a_1\dots a_k$ form a basis, this subdivision is a triangulation. A concrete example is depicted in~Figure~\ref{fig:lifting}.

\begin{figure}[htpb]
\begin{center}
\centering
\begin{tabular}{m{.2\textwidth}m{.2\textwidth}m{.2\textwidth}}
 \includegraphics[width=.2\textwidth]{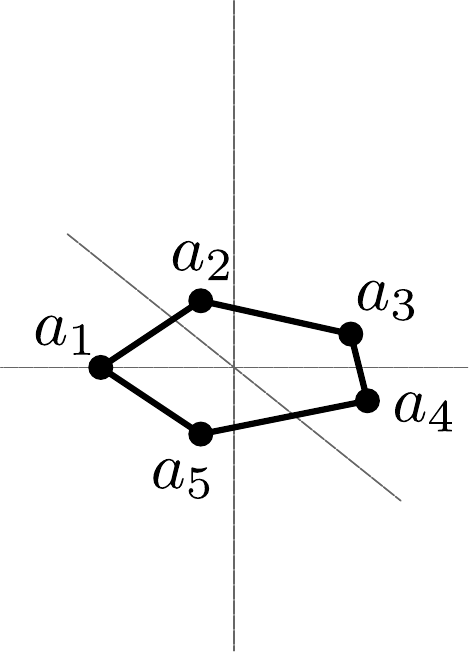}&
 \includegraphics[width=.2\textwidth]{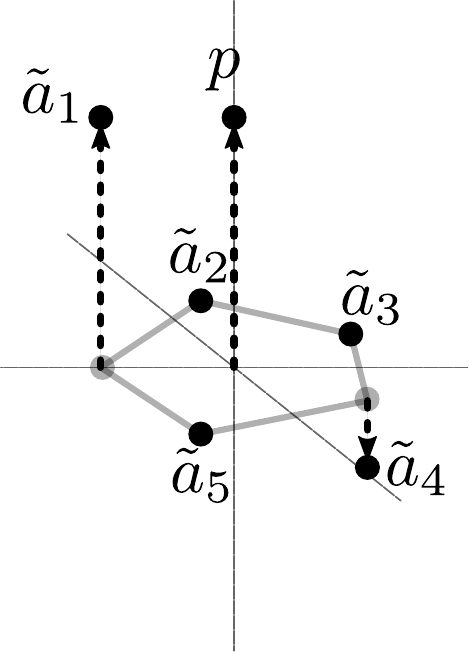}&
 \includegraphics[width=.2\textwidth]{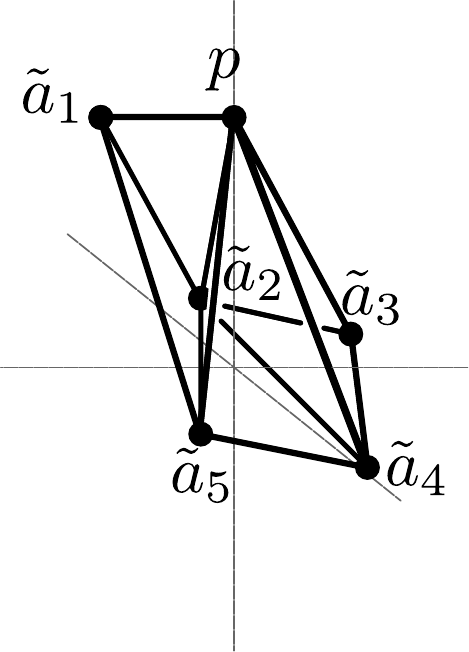}\\
$P=\conv(A)$&Lifting $A$ to $\tilde A$&$\tilde P=\conv(\tilde A)$
\end{tabular}
{\begin{tabular}{m{.22\textwidth}}
 \centering\includegraphics[width=.12\textwidth]{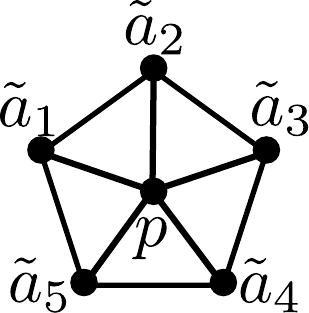}\\
Upper envelope of~$\tilde P$\\
 \centering\includegraphics[width=.12\textwidth]{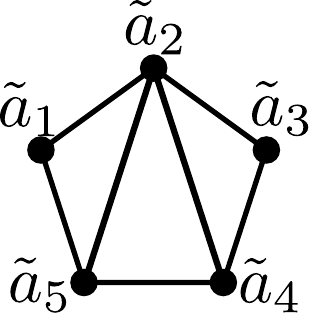}\\
Lower envelope of~$\tilde P$
\end{tabular}}
\end{center}
 \caption{The lifting of a pentagon $P=\conv({A})$ to ${\tilde P}=\conv({\tilde A})$ when $\Gale{{\tilde A}}=\Gale{A}[p]$ and $p=[a_1^-,a_4^+]$. Its upper envelope are pyramids over facets of $P$, while the lower envelope is the lexicographic triangulation of $P$ on $[a_1^+,a_4^-]$.}
 \label{fig:lifting}
\end{figure}

Our formulation of the definition of lexicographic subdivision is based on~\cite{DeLoeraRambauSantosBOOK}. However we use a different ordering, the same as in~\cite{Santos2002}, that mirrors the definition of lexicographic extension (with opposite signs). See also~\cite{Lee1991}.

\begin{definition}
Let $P$ be a $d$-polytope with $n$ vertices $a_1,\dots,a_n$. The \emph{lexicographic subdivision} of $P$ on $[a_1^{\ep_1},a_2^{\ep_2},\dots,a_k^{\ep_k}]$, where $\ep_i=\pm 1$, is defined recursively as follows.
\begin{itemize}
 \item If $\ep_1=+1$ (\emph{pushing}), then the lexicographic subdivision of $P$ is the union of the lexicographic subdivision of $P\setminus a_1$ on $[a_2^{\ep_2},\dots,a_k^{\ep_k}]$, and the simplices joining $a_1$ to the (lexicographically subdivided) faces of $P\setminus a_1$ visible from it.
 \item If $\ep_1=-1$ (\emph{pulling}), then the lexicographic subdivision of $P$ is the unique subdivision in which every maximal cell contains $a_1$ and which, restricted to each proper face $F$ of $P$, coincides with the lexicographic subdivision of that face on $[a_2^{\ep_2},\dots,a_k^{\ep_k}]$.
\end{itemize}
\end{definition}

\begin{remark}
 The resemblance with Sanyal and Ziegler's description of the vertex figures of the neighborly cubical polytopes in~\cite{SanyalZiegler2010} is not a coincidence. Indeed, the Gale duals of those vertex figures are lexicographic extensions of the dual of a fixed neighborly polytope.
\end{remark}

\begin{remark}\label{rmk:inscribable}
 The inscribability of the neighborly polytopes in $\cG$ can be proved with this primal interpretation of Gale Sewing. For this, the key observation in~\cite{GonskaPadrol} is that the pushing triangulation induced by the Double Extension Theorem~\ref{thm:thethm} can always be realized as a Delaunay triangulation.
\end{remark}

\section{Comparing and combining the constructions}\label{sec:comparing}

In this section we compare and combine the construction techniques for neighborly polytopes, which are strongly related. 

\subsection{Extended Sewing and Omitting is included in Gale Sewing}

Our first goal is to prove Corollary~\ref{cor:cOsubsetcG}, that states that if a neighborly polytope~$P$ is built via Extended Sewing and Omitting (Construction~\ref{constr:cO}), then~$P$~can also be built with Gale Sewing (Construction~\ref{constr:cG}).
For that we will need the following theorem, which implies that the contraction and deletion of an element determine an oriented matroid up to the reorientation of that element.

\begin{theorem}[{\cite[Theorem 4.1]{RichterGebertZiegler1994}}]\label{thm:minorsfixmatroid}
Let $\cM'$ and $\cM''$ be two oriented matroids with the same ground set $E$, of respective ranks $\rd$ and $\rd-1$, 
such that $\cov(\cM'')\subseteq\cov(\cM')$. Then there is an oriented matroid $\cM$ with ground set $E\cup\{p\}$ that fulfills $\cM\setminus p =\cM'$ and $\cM/p=\cM''$. The oriented matroid $\cM$ has rank $\rd$ and is unique up to reorientation of $p$.
\end{theorem}

\begin{corollary}\label{cor:minorsfixmatroid}
Let $\cM$ and $\cM$ be oriented matroids on a ground set $E$. If $\cM\setminus p= \cM'\setminus p$ and $\cM/ p= \cM'/ p$, then $\cM$ and $\cM'$ coincide up to the reorientation of $p$.

If additionally there is an element $q\in E$ and some $\alpha=\pm 1$ such that $p$ and $q$ are $\alpha$-inseparable in both $\cM$ and $\cM'$, then $\cM= \cM'$.
\end{corollary}

 \begin{theorem}\label{thm:cOsubsetcG}
For any uniform neighborly matroid $\tilde \cP$ in $\cE$ of rank $\rd$ with $n$ elements, and any universal flag $\tilde \cF=\{\tilde F_i\}_{i=1}^{m}$ of $\tilde \cP$ derived from Remark~\ref{rmk:extnewuniflags}, where $m=\ffloor{\rd-1}{2}$ and $\tilde F_j=\bigcup_{i=0}^{j-1} \{\tilde p_{m-i}, \tilde q_{m-i}\}$, there is a sequence of balanced matroids $\tilde \cM_k$, for $0\leq k\leq m$, such that:
 \begin{enumerate}
  \item $\tilde \cM_m=\Gale {\tilde \cP}$,
  \item $\tilde \cM_0$ has rank $\rr=n-\rd$ and $n-2m$ elements, and
  \item for $0<k\leq m$,  $\tilde \cM_k=\tilde \cM_{k-1}[\tilde p_{k}][\tilde q_{k}]$ is a double extension as in Theorem~\ref{thm:thethm}.
 \end{enumerate}
 \end{theorem}

 \begin{proof}
 The proof is by induction on $n$. The base case is when $n=\rd$. Then both~$\Gale {\tilde \cP}$ and $\tilde \cM_m$ have rank~$0$, and the claims follow trivially. 
 
 If $n\geq \rd$, then $\tilde \cP=\cP[\cF']$ where an element $p$ is sewn onto some~$\cP\in\cE$ through some flag~$\cF'$ that contains a universal subflag $\cF$. By induction hypothesis we can assume that $\cP$, $\cF$ and $\cF'$ fulfill:
 \begin{itemize}
  \item $\cP$ has rank $\rd$ and $n-1$ elements. Its dual $\Gale\cP$ equals $\cM_m$ for a sequence of matroids $\cM_k$ for $0\leq k\leq m$ constructed as follows:
    $\cM_0$ is a uniform balanced matroid of rank $\rr=n-\rd-1$ and $n-2m-1$ elements, and for $0< k\leq m$ \begin{equation}\label{eq:defMk}\cM_k:=\cM_{k-1}[p_k][q_k],\end{equation} for lexicographic extensions defined by
\begin{align}
p_k&:=[a_{k1}^{\ep_{k1}},\dots,a_{kr}^{\ep_{kr}}],& q_k&:=[p_k^-,a_{k1}^{-},\dots,a_{k(r-1)}^{-}];\label{eq:defpkqk}
\end{align}
where the $a_{ij}$ are pairwise distinct elements of $\cM_{i-1}$. 

\item $\cF$ is of the form $\cF=\{F_i\}_{i=1}^{m}$, where $F_j=\bigcup_{i=0}^{j-1} \{p_{m-i},q_{m-i}\}$ (that is, $F_{m-k}=\{p_m,q_m,\dots,p_{k+1},q_{k+1}\}$). 

\item The flag $\cF'$ contains $\cF$ as a subflag. By Lemma~\ref{lem:galesewingorder} we assume without loss of generality that all split faces in $\cF'$ are $q_i$-split.
 \end{itemize}

The proof needs some further notation. Let $\cP_k:=\cP/F_{m-k}$ for $k=0,\dots,m$, and observe that $\cP_k=\Gale{\cM_k}$, for all~$k$, by deletion-contraction duality. Moreover, we define the sets $\tilde F_{j+1}$, all containing the sewn element $p$, as $\tilde F_{j+1}:=F_{j}\cup q_{m-j}\cup p$ (that is $\tilde F_{m-k}=\{p_m,q_m,\dots,p_{k+2},q_{k+2},q_{k+1},p\}$). We denote $\tilde \cP_k=\tilde \cP/\tilde F_{m-k}$ and observe that by Lemma~\ref{lem:quotientsofextShemersewing}, $\tilde \cP_k= \cP_k[\cF'/F_{m-k}]$. Here in $\cP_k[\cF'/F_{m-k}]$, the sewn vertex is $p_{k+1}$, and thus
$\tilde \cP_k\setminus p_{k+1}= \cP_k$.  We occasionally abbreviate $p=p_{m+1}$.

Now, set $\tilde \cM_0=\Gale{\cP_0}$ and for $0<k\leq m$ let 
\begin{align*}
 q_k&:=
\begin{cases} [ p_k^+,{( a_{k1})}^{-\ep_{k1}},\dots,{( a_{kr})}^{-\ep_{kr}}]& \text{if $F_{m-k+1}$ is not split in $\cF'$,}\\
[ p_k^-,{( a_{k1})}^{\ep_{k1}},\dots, {( a_{kr})}^{\ep_{kr}}]& \text{if $F_{m-k+1}$ is split in $\cF'$.}
\end{cases}\\
  p_{k+1}&:=[ q_k^-,p_k^-,{( a_{k1})}^{-},\dots,{( a_{k(r-1)})}^{-}],\\
 \tilde \cM_{k}&:=\tilde \cM_{k-1}[ q_k][ p_{k+1}].
\end{align*}

With this notation, we claim that 
\(
               \tilde \cM_k=\Gale{\tilde\cP_k}
\) 
 (cf. Figure~\ref{fig:ShemerVsMe}).
We prove this claim by induction on $k$, and the base case $k=0$ is true by construction. 

\begin{figure}[htpb]
\begin{center}

\begin{tabular}{m{.18\textwidth}m{.1\textwidth}m{.18\textwidth}m{.37\textwidth}}
\includegraphics[width=.2\textwidth]{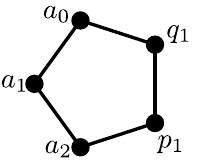}&$ \xleftrightarrow[\Gale{\cP}=\cM_1]{\text{  duality  }}$&
\includegraphics[width=.2\textwidth]{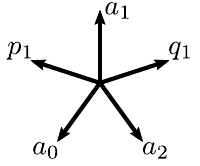}
&
$\cM_0=\{a_0,a_1,a_2\}$\newline $\cM_1=\cM_0[a_2^-,a_1^-][p_1^-,a_2^-]$
\\
\centering$\cP$\\
\centering {\rotatebox{90}{$\xleftarrow[{\, p=[q_1^+,p_1 ^+,a_1^-]\,}]{\text{sewing}}$}} &&&
\\
\centering$\tilde \cP=\cP[p]$\\
\includegraphics[width=.2\textwidth]{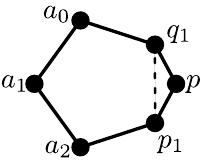}
&$\xleftrightarrow[\Gale{\tilde \cP}\simeq\tilde \cM_1]{\text{  duality  }}$
&\includegraphics[width=.2\textwidth]{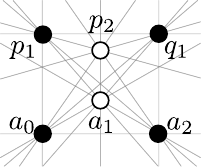}
&$\tilde{\cM}_0=\{{p}_1,{a}_0,{a}_1,{a}_2\}$\newline 
$\tilde{\cM}_1=\tilde{\cM}_0[{p}_1^+,{a}_2^+,{a}_1^+][\tilde{q}_1^-,{p}_1^-,{a}_2^-]$
\end{tabular}
\end{center}
\caption{We reach the lower left figure by two paths: First (starting in the top right), $\cM$~is constructed from $\cM_0$ after Gale Sewing $p_1=[a_2^-,a_1^-]$ and $q_1=[p_1^-,a_2^-]$. The dual of $\cM_1$ is $\cP$ (top left). Then $\tilde{\cP}$ (lower left) is constructed from~$\cP$ by sewing $p$ onto the flag formed by the universal edge $\{p_1,q_1\}$ (which is not split). 
\newline
In the second path (lower right), $\tilde{\cM_1}$~is constructed from $\tilde\cM_0$ by Gale Sewing $\tilde{q}_1=[{p}_1^+,{a}_2^+,{a}_1^+]$ and ${p}_2=[\tilde{q}_1^-,{p}_1^-,{a}_2^-] $; then we dualize to get $\Gale{\tilde{\cM_1}} =\tilde \cP$.
}
 \label{fig:ShemerVsMe}
\end{figure}

Let $k>0$ and assume that $\tilde \cM_{k-1}=\Gale{\tilde\cP_{k-1}}$. The proof uses Corollary~\ref{cor:minorsfixmatroid}
twice 
and relies on the following facts (our claim is the final fact \ref{it:Mk}):
\begin{enumerate}
 [label=\textbf{(\Alph*)}]
 
\item\label{it:Mkmodpk+1} \textbf{\boldmath$\tilde \cM_k/p_{k+1}=\Gale{\tilde\cP_k}/p_{k+1}$.}

 Since by definition $\tilde\cP_k\setminus p_{k+1}=\cP_k$, then  $\Gale{\tilde\cP_k}/p_{k+1}=\Gale{\cP_k}=\cM_k$ and we only need to prove that \begin{equation}\label{eq:quotientsofGaleareShemer}\tilde\cM_k/p_{k+1}= \cM_k.\end{equation}
By Lemma~\ref{lem:quotientsofGalesewnareGalesewn}, 
 $(\tilde \cM_k/p_{k+1})= (\tilde \cM_{k-1}/p_{k})[\tilde a_{k1}^{\ep_{k1}},\dots,\tilde a_{kr}^{\ep_{kr}}][{x'_k}^-,\tilde a_{k1}^{-},\dots,\tilde a_{k(r-1)}^{-}]$.
Then we get \eqref{eq:quotientsofGaleareShemer} combining that $\tilde \cM_{k-1}/p_k= \cM_{k-1}$ (by the induction hypothesis) with the equations \eqref{eq:defMk} and \eqref{eq:defpkqk} that define $\cM_k$.

 \item \label{it:Mkinsep} \textbf{\boldmath$p_{k+1}$ and $q_k$ are $(+1)$-inseparable in $\tilde\cM_k$ and~$\Gale{\tilde\cP_k}$.}

  Follows from Lemma~\ref{lem:leinseparable} and the definitions of $\tilde\cM_k$ and~${\tilde\cP_k}$.

 \item\label{it:Mk-pk+1modqk} \textbf{\boldmath$(\Gale {\tilde \cP_k}\setminus p_{k+1})/ q_{k}=(\tilde \cM_k\setminus p_{k+1})/ q_{k}$.}

 By Lemma~\ref{lem:contractdeletele},
 \[
(\tilde \cM_k\setminus p_{k+1})/ q_k=(\tilde \cM_k /  p_{k+1})\setminus  q_k, \text{ and }  
 (\Gale{\tilde\cP_k}\setminus p_{k+1})/q_k=(\Gale{\tilde\cP_k}/p_{k+1})\setminus q_k.
 \]
Now $(\Gale {\tilde \cP_k}/ p_{k+1})\setminus q_{k}=(\tilde \cM_k/ p_{k+1})\setminus q_{k}$ follows directly from \ref{it:Mkmodpk+1}.

  \item \label{it:Mk-pk+1-qk} \textbf{\boldmath$(\Gale {\tilde \cP_k}\setminus p_{k+1})\setminus q_{k}=(\tilde \cM_k\setminus p_{k+1})\setminus q_{k}$.}
 
 This is direct by the induction hypothesis, since
 \[(\tilde \cM_k\setminus p_{k+1})\setminus q_k =\tilde \cM_{k-1}=\Gale{\tilde\cP_{k-1}}=\Gale{(\tilde\cP_{k}/\{q_k,p_{k+1}\})}=(\Gale{\tilde\cP_{k}}\setminus p_{k+1})\setminus q_k.\]

\item\label{it:Mk-pk+1insep} 
 \textbf{\boldmath$q_k$ and  $p_k$ are $\alpha$-inseparable in $\tilde \cM_k\setminus p_{k+1}$ and $\Gale{(\tilde\cP_{k}/ p_{k+1})}$, where $\alpha:=-1$ if $F_{m-k+1}$ is not split and $\alpha:=+1$ otherwise.}

If $F_{m-k+1}$ is not split then $\tilde q_k$ is $(-1)$-inseparable with $\tilde p_k$ in $\tilde \cM_k\setminus p_{k+1}=\tilde \cM_{k-1}[q_k]$ by construction. 
Moreover, by Proposition~\ref{prop:allquotientsofle} \[\tilde\cP_{k}/ p_{k+1}=\big(\cP_{k}\underbrace{[\cF'/F_{m-k}]}_{p_{k+1}}\big)/p_{k+1}= \big(\cP_{k}/ p_{k}\big)\underbrace{[q_k^-,\dots]}_{p_k}.\]
In this last expression the sewn vertex is $p_k$, which is $(+1)$-inseparable from~$q_k$ by Lemma~\ref{lem:leinseparable}. This means that $q_k$ is $(-1)$-inseparable with $p_k$ in $\Gale{(\tilde\cP_{k}/ p_{k+1})}$ because of Lemma~\ref{lem:insep}.

The proof for the case when $F_{m-k+1}$ is $q_k$-split is analogous.
  
\item \label{it:Mk-pk+1} \textbf{\boldmath $\tilde \cM_k\setminus p_{k+1}=\Gale{\tilde\cP_k}\setminus p_{k+1}$.}

This is a direct consequence of Corollary~\ref{cor:minorsfixmatroid} by \ref{it:Mk-pk+1-qk}, \ref{it:Mk-pk+1modqk} and \ref{it:Mk-pk+1insep}.

\item\label{it:Mk} \textbf{\boldmath $\tilde \cM_k\simeq\Gale{\tilde\cP_k}$.}

This follows also from Corollary~\ref{cor:minorsfixmatroid} by \ref{it:Mkmodpk+1}, \ref{it:Mk-pk+1} and \ref{it:Mkinsep}.
\end{enumerate}

We have already seen that $\tilde \cP=\tilde \cP_m$ is Gale sewn, but we have to test our complete induction hypothesis. Namely, it remains to be checked that for each universal flag of $\tilde \cP$ obtained by Remark~\ref{rmk:extnewuniflags}, $\tilde \cP$ can be obtained 
by Gale Sewing the elements in the order marked by the flag. This is a consequence of Lemma~\ref{lem:galesewingorder}, which allows to change the order of the sewings in $\tilde\cM_m$. We omit the details of this easy computation that concludes the proof of Theorem~\ref{thm:cOsubsetcG}.
\end{proof}

\begin{corollary}\label{cor:cOsubsetcG}
 $\cO\subseteq\cG$.
\end{corollary}
\begin{proof}
 By Proposition~\ref{prop:allquotientsofGalesewnareGalesewn}, to prove $\cO\subseteq\cG$ it suffices to see that $\cE\subseteq \cG$. This follows directly from Theorem~\ref{thm:cOsubsetcG}.
 
 Indeed, let $\tilde \cP\in \cE$. With the notation of Theorem~\ref{thm:cOsubsetcG}, if $\rd$ is odd, then $\tilde \cM_0$ is balanced of rank $r$ with $r+1$ elements, which implies that it is the oriented matroid of $\{e_1, \dots, e_r, -\sum_{i=1}^r e_i\}$. Therefore, $\tilde \cP$ is in $\cG$ because it is built using Construction~\ref{constr:cG}. If $\rd$ is even, then $\tilde \cP$ is in $\cG$ in the sense of Remark~\ref{rmk:odddim}.
\end{proof}

\begin{remark}\label{rmk:doesnotgeneralize}
The fact that $\cE\subsetneq\cG$ implies that in some sense Gale Sewing generalizes ordinary (Extended) Sewing. However, it is not true that the Extended Sewing Theorem~\ref{thm:shemersewing} is a consequence of the Gale Sewing Theorem~\ref{thm:thethm}, because there are neighborly matroids that have universal flags but are not in~$\cG$. Hence one can sew on them but they 
cannot be treated with Theorem~\ref{thm:cOsubsetcG}. 
This will become clear in Section~\ref{sec:nonrealizable}, where we work with $\cM^{10}_{425}$, a non-realizable neighborly matroid that has universal flags. Since Gale Sewing (Construction~\ref{constr:cG}) only builds realizable matroids, this matroid is not in $\cG$ and yet one can sew on it. This shows why both constructions are needed.
\end{remark}

\subsection{Some exact numbers}

We have worked with five families of neighborly polytopes:

\begin{description}
 \item[$\cN$] All neighborly polytopes.
 \item[$\cS$] Totally sewn neighborly polytopes (Sewing, Construction~\ref{constr:cS}).
 \item[$\cE$] Neighborly polytopes constructed by Extended Sewing (Construction~\ref{constr:cE}).
 \item[$\cO$] Neighborly polytopes built by Extended Sewing and Omitting (Construction~\ref{constr:cO}).
 \item[$\cG$] Gale sewn neighborly polytopes (Construction~\ref{constr:cG}).
\end{description}

Table~\ref{tb:numcombtypes} contains the exact number of (unlabeled) combinatorial types of $d$-dimensional neighborly polytopes with $n$ vertices in each of these families for the cases $d=4$ and $n=8,9$ and for $d=6$ and $n=10$. Exact numbers for $\cN$ come from~\cite{AltshulerSteinberg1973} and~\cite{BokowskiShemer1987}, exact numbers for $\cS$ and $\cO$ come from~\cite{Shemer1982}. Numbers for $\cG$ and $\cE$ have been computed with the help of \texttt{polymake}~\cite{polymake}.

\begin{table}[htpb]
\centering
  \caption{Exact number of combinatorial types}	\label{tb:numcombtypes}
  \setlength{\tabcolsep}{15pt}
   \renewcommand{\arraystretch}{1.3}
  \begin{tabular}{ c  c c c  c  c  c  c }
    \hline
      $d$& $n$ && $\cS$ & $\cE$ & $\cO$ & $\cG$ & $\cN$\\
    \hline 
      4 & 8 && 3 & 3 & 3 & 3 & 3 \\ 
      4 & 9 && 18 & 18 & 18 & 18 & 23 \\ 
      6 & 10 && 15 & 26 & 28 & 28 & 37 \\ 
    \hline
  \end{tabular}
\end{table}

In view of Table~\ref{tb:numcombtypes}, the known relationships  between these families are summarized in the following proposition.
\begin{proposition}
\(\cS\subsetneq\cE\subsetneq\cO\subseteq\cG\subsetneq\cN.\) 
\end{proposition}

This begs the question:

\begin{question}
Is $\cO=\cG$?
\end{question}

\subsection{Non-realizable neighborly oriented matroids}\label{sec:nonrealizable}
Since the only neighborly matroids of rank~$3$ are cyclic polytopes, there are no non-realizable neighborly matroids of rank~$3$. The sphere ``$\cM^{10}_{425}$'' from Altshuler's list~\cite{Altshuler1977} corresponds to a neighborly matroid of rank~$5$ with $10$ elements. In~\cite{BokowskiGarms1987}, this matroid is shown to be non-realizable, thus proving that non-realizable neighborly matroids exist. Kortenkamp's construction~\cite{Kortenkamp1997} can also be used to build non-realizable neighborly matroids of corank $3$. We combine Theorems~\ref{thm:extshemersewing} and~\ref{thm:thethm} to show that there are many non-realizable neighborly matroids. A lower bound for the cardinality of the number of non-realizable neighborly matroids is derived later in Theorem~\ref{thm:nonrealizablebound}.

\begin{theorem}\label{thm:nonrealizable}
 There exists a non-realizable neighborly matroid of rank $\rd$ with $n$ elements for every $\rd\geq 5$ and $n\geq \rd+5$.
\end{theorem}
\begin{proof}
We start with $\cM^{10}_{425}$. With the vertex labeling of \cite{Altshuler1977}, $\{0,1\}$, $\{2,3\}$, $\{4,5\}$, $\{6,7\}$ and $\{8,9\}$ are universal edges of $\cM^{10}_{425}$ because the corresponding contractions are polygons with $8$ vertices. In particular, $\{0,1\}\subset \{0,1,2,3\}$ is a universal flag. Hence, applying the Extended Sewing Theorem~\ref{thm:extshemersewing} we get many non-realizable matroids of rank~$5$ with $n$ vertices for any $n\geq 10$.

Now, applying to these matroids the Corollary~\ref{cor:primalGaleSewing} of the Gale Sewing Construction, we get non-realizable oriented matroids of rank $5+2k$ and $n$ vertices for any $k\geq 0$ and any $n\geq 10+2k$.

To get non-realizable matroids of even rank, just observe that any single element extension on the dual of a neighborly matroid of rank $2k+1$ yields the dual of a neighborly matroid of rank $2k+2$.
\end{proof}

All neighborly matroids of rank $2m+1$ that have $n\leq2m+3$ vertices are cyclic polytopes. Moreover, all oriented matroids of rank $5$ with $8$ elements are realizable~\cite[Corollary 8.3.3]{OrientedMatroids1993}. 
Hence the first case (of odd rank) that Theorem~\ref{thm:nonrealizable} does not deal with are neighborly matroids of rank $5$ with $9$ elements.

\section{Many neighborly polytopes}\label{sec:counting}

The aim of this section is to find lower bounds for {$\lnei{n}{d}$}\index{$\lnei{n}{d}$}, the number of combinatorial types of vertex-labeled neighborly polytopes with $n$ vertices in dimension~$d$. 
Since two neighborly polytopes with the same combinatorial type have the same oriented matroid (Theorem~\ref{thm:neigharerigid}), it suffices to bound the number of labeled realizable neighborly matroids. 

Our strategy will consist in using the Gale Sewing technique of Theorem~\ref{thm:thethm} to construct many neighborly polytopes in~$\cG$ for which we can certify that their oriented matroids are all different.

We only deal with polytopes and oriented matroids that are labeled. 
Nevertheless, our bounds are so large as to present the same kind of growth as 
the naive bounds for unlabeled combinatorial types obtained by dividing by $n!$. Namely, $\frac{\lnei{n}{d}}{n!}\geq n^{\frac{d-2}{2} n(1+o(1))}$ for fixed dimension~$d>2$ and $n\rightarrow \infty$.

\subsection{Many lexicographic extensions}
A first step is to compute lower bounds for $\lle{n}{r}$, the smallest number of different labeled lexicographic extensions that any balanced matroid of rank~$r$ with $n$ elements must have. Here, a labeled lexicographic extension of~$\cM$ is a lexicographic extension $\cM[p]$ labeled in such a way that the labels of the elements of~$\cM$ are preserved.

There are $2^{r}\!\frac{n!}{(n-r)!}$ different expressions for lexicographic extensions of a rank~$r$ oriented matroid on $n$ elements, yet not all of them represent different labeled oriented matroids. We aim to avoid counting the same extension twice with two different expressions.

\begin{proposition}\label{prop:lble}
 Let $\cM$ be a rank $r>1$ labeled uniform balanced matroid with $n$ elements. If $n-r-1\ge 2$ is even, then there are at least \begin{equation}\label{eq:bndle1}\lle{n}{r}\geq\frac{2n!}{(n-r+1)!}
\end{equation}
different uniform labeled lexicographic extensions of $\cM$.
\end{proposition}
\begin{proof}
We focus only on those extensions where $\ep_i=+$ for all $i$, and show that they are unambiguous except for the last element $a_r^{\ep_r}$.

For this, observe that if $r>1$ and the lexicographic extensions by $[a_1^{+},\dots,a_r^{+}]$ and $[{a'_1}^{+},\dots,{a'_r}^{+}]$ yield the same oriented matroid,  then for every cocircuit $C\in\co(\cM)$ with $C(a_1)\neq 0$ and $C(a_1')\neq 0$, the signature $\sigma:\co(\cM)\rightarrow\{\pm,0\}$ of the lexicographic extension fulfills $\sigma(C)=C(a_1)$ and $\sigma(C)=C(a_1')$. Thus, if $a_1$ and $ a_1'$ are different, then $a_1$ and $a_1'$ are $(+1)$-inseparable in $\Gale\cM$ and hence, by Lemma~\ref{lem:insep}, $(-1)$-inseparable in~$\cM$.

But balanced matroids of rank $r\geq 2$ and even corank only have $(+1)$-inseparable pairs (see Lemma~\ref{lem:balonlycovar}), which proves that $a_1=a_1'$. Analogously, if $a_i$ and $a_i'$ are the first distinct elements and $i<r$, we can apply the previous argument on the contraction by $\{a_1,\dots,a_{i-1}\}$.

Hence, there are at least $\frac{n!}{(n-r+1)!}$ different choices for the first $r-1$ elements (which give rise to different matroids). For the last element, observe that $\cM/\{a_1,\dots,a_{r-1}\}$ is a matroid of rank~$1$, and that there are exactly two possible different extensions for a matroid of rank~$1$.
\end{proof}

\begin{remark}\label{rmk:lble2}
 In the bound \eqref{eq:bndle1}, we lose a factor of up to $2^{r-1}$ from the real number. This factor is asymptotically much smaller than our bound of $\frac{2n!}{(n-r+1)!}$. 

 In fact, it is not difficult to prove that $\lle{n}{r}\geq2^{r-1}\frac{n!}{(n-1)(n-r)!}$ by giving some cyclic order to the elements of $\cM$ and counting only the lexicographic extensions $[a_1^{\ep_1},\dots, a_r^{\ep_r}]$ that fulfill
\begin{enumerate}[label={(\roman*)}, leftmargin=*]
 \item For $1<i<r$, $a_i^{\ep_i}$ is not $b^{-\alpha \ep_{i-1}}$ if $b<a_{i-1}$ and $b$ and $a_{i-1}$ are $\alpha$-inseparable in $\cM/\{a_1,\dots,a_{i-2}\}$.
 \item For $1<i<r$, $a_i^{\ep_i}$ is not $c^{\alpha\beta \ep_{i-1}}$ when there exists $b$ with $c<b<a_{i-1}$ such that $b$ and $a_{i-1}$ are $\alpha$-inseparable in $\cM/\{a_1,\dots,a_{i-2}\}$ and $c$ is $\beta$-inseparable from $b$ in $\cM/\{a_1,\dots,a_{i-1}\}$.
 \item $a_r$ and $a_{r-1}$ are $\alpha$-inseparable in $\cM/\{a_1,\dots,a_{r-2}\}$, $a_{r-1}>a_r$ and $\ep_r=\alpha \ep_{r-1}$.
\end{enumerate}
 But then the formulas become more complicated and add nothing substantial to the result.
\end{remark}
\begin{remark}
 The hypothesis of balancedness in not necessary in Proposition~\ref{prop:lble} and Remark~\ref{rmk:lble2}, and one can adapt the proofs to obtain lower bounds for the number of lexicographic extensions that any oriented matroid must have.
\end{remark}

\subsection{Many neighborly polytopes in $\cG$}

Once we have bounds for $\lle{n}{r}$, we can obtain bounds for $\lnei{n}{d}$ using the Gale Sewing Construction. But first we do a case where we know the exact number.

\begin{lemma}\label{lem:nlpolygons}
 The number of labeled balanced matroids of rank $r$ with $r+3$ elements is $\frac12{(r+2)!}$.
\end{lemma}
\begin{proof}
 Balanced matroids of rank $r$ with $r+3$ elements are dual to polygons with $r+3$ vertices in $\RR^2$. 
There are clearly $\frac12{(r+2)!}$ different combinatorial types of labeled polygons with $r+3$ vertices.
\end{proof}

For our next proof, we need the following result concerning the \emph{inseparability graph} $\IG(\cM)$ of an oriented matroid~$\cM$, which is defined to be the graph that has the elements of $\cM$ as vertices and the pairs of inseparable elements as edges.
\begin{theorem}[{\cite[Theorem 1.1]{CordovilDuchet1990}}]\label{thm:IG}
 Let $\cM$ be a rank $r$ uniform oriented matroid with $n$ elements.
\begin{itemize}
 \item If $r\leq 1$ or $r\geq n-1$, then $\IG ( \cM )$ is the complete graph $K_n$.
 \item If $r=2$ or $r=n-2$, then $\IG(\cM)$ is an $n$-cycle.
 \item If $2<r<n-2$, then $\IG(\cM)$ is either a $n$-cycle, or a disjoint union of chains.
\end{itemize}
\end{theorem}

\begin{lemma}\label{lem:lbbm}
For $r\geq2$ and $m\geq2$, the number of labeled balanced matroids of rank $r$ with $r+1+2m$ elements is  $\lnei{2m+r+1}{2m}$ and fulfills
\begin{equation}\label{eq:lbbm}
\lnei{2m+r+1}{2m}\geq \lnei{2m+r-1}{2m-2}\frac{r+2m}{2}\lle{r+2m-1}{r}.
\end{equation}
\end{lemma}
\begin{proof}
 The characterization is direct by duality. For the bound, choose a balanced matroid~$\cM$ of rank $r$ with $r+1+2(m-1)$ elements such that each element has a label in the set $\{1,\dots,r+1+2(m-1)\}$. And let $\cM[p]$ be a labeled lexicographic extension of $\cM$ by $p=[a_1^{\ep_1},\dots,a_r^{\ep_r}]$. Finally let $\cM[p][q]$ be the extension of $\cM[p]$ by $q=[p^-,a_1^{-},\dots,a_{r-1}^{-}]$, which is balanced by the Double Extension Theorem~\ref{thm:thethm}. 
 
 We consider all the relabelings of $\cM[p][q]$ such that $q$ gets label $r+2m+1$ and the labeling of $\cM[p][q]$ on~$\cM$ preserves the relative order of the original labeling of~$\cM$. 
 
 We claim that each labeled matroid obtained this way is constructed at most twice. Indeed, observe that $p$ and $q$ are inseparable because of Lemma~\ref{lem:leinseparable}. Moreover, by Theorem~\ref{thm:IG}, $q$ is inseparable from at most two elements in $\cM[p][q]$ because $2\le r$ and $2\le m$. Since the label of $q$ is fixed, $\cM[p][q]$ might have been counted twice if $q$ is inseparable from two elements in $\cM[p][q]$. 
 
 Summing up, we can choose among $\lnei{2m+r-1}{2m-2}$ matroids $\cM$, $\lle{r+2m-1}{r}$ extensions $\cM[p]$, and $(r+2m)$ labels for $p$ to construct at least \[(r+2m)\lnei{2m+r-1}{2m-2}\lle{r+2m-1}{r}\] labeled balanced oriented matroids, where each matroid is counted at most twice. This yields the claimed formula. 
\end{proof}

This result allows us to give our first explicit lower bound on the number of neighborly polytopes. 
\begin{proposition}\label{prop:brutebound}
 The number of labeled neighborly polytopes in even dimension $d=2m\geq 2$ with $n=r+d+1$ vertices fulfills 
\begin{equation}\label{eq:lblnei1}
  \lnei{2m+r+1}{2m}\geq\prod_{i=1}^{m} {\frac{(r+2i)!}{(2i)!}}
\end{equation}
\end{proposition}
\begin{proof}
Observe that by rigidity (Theorem~\ref{thm:neigharerigid}), counting labeled neighborly polytopes is equivalent to counting labeled neighborly oriented matroids. By duality, this is in turn equivalent to counting balanced oriented matroids. This we do.

Lemma~\ref{lem:nlpolygons} proves the required formula in the initial case $m=1$, and yields $\lnei{2+\rr+1}{2}=\frac12(\rr+2)!$. For $m\ge2$, we observe that by Proposition~\ref{prop:lble},
\[\frac{r+2m}{2}\lle{r+2m-1}{r}\geq 
\frac{(\rr+2m)!}{(2m)!}.\]
Finally, we apply Lemma~\ref{lem:lbbm} to obtain~\eqref{eq:lblnei1}.
\end{proof}

Although Proposition~\ref{prop:brutebound} provides us with the desired bound, it is hard to understand its order of magnitude at first sight. This is the reason why we present the following simplified bound~\eqref{eq:thebound}.
\begin{theorem}\label{thm:lblnei}
The number of labeled neighborly polytopes in even dimension $d$ with $n$ vertices fulfills 
\begin{equation}\tag{\ref{eq:thebound}}
\lnei{r+d+1}{d}\geq \frac{\left( r+d \right) ^{\left( \frac{r}{2}+\frac{d}{2} \right) ^{2}}}{{r}^{{(\frac{r}{2})}^{2}}{d}^{{(\frac{d}{2})}^{2}}{\e^{3\frac{r}{2}\frac{d}{2}}}},
\end{equation}
that is, 
\begin{equation*}
 \lnei{n}{d}\geq \frac{\left(n-1\right) ^{\left( \frac{n-1}{2}\right) ^{2}}}{{(n-d-1)}^{{\left(\frac{n-d-1}{2}\right)}^{2}} d  ^{{(\frac{d}{2})}^{2}}{\e^{\frac{3d(n-d-1)}{4}}}}.
\end{equation*}
\end{theorem}
\begin{proof}
We start from Equation~\eqref{eq:lblnei1}, and bound the natural logarithm of $\lnei{r+1+2m}{2m}$. Using the fact that $\int_{a-1}^{b} f(s)\  \mathrm{d}s \le \sum_{i=a}^{b} f(i)$ for any increasing function $f$, we obtain
\begin{align*}
\ln\left( \lnei{r+1+2m}{2m}\right)\geq& \ln\left(\prod_{i=1}^{m} \frac{(\rr+2i)!}{(2i)!}
\right)=\sum_{i=1}^{m} {\sum_{j=1}^{r}{ \ln\left(2i+j\right)}}\\
\geq&\int_{i=0}^{m} {\int_{j=0}^{r}{ \ln\left(2i+j\right) \mathrm{d}j} \mathrm{d}i}\\
=&  \frac{\left( 2m+r \right) ^{2}\ln  \left( 2m+r
 \right)-{r}^{2}\ln  \left( r \right)}{4} -{m}^{2} \ln  \left( 2m
 \right) -\frac{3mr}{2}.
\end{align*}
Hence
\begin{align*}
\lnei{r+1+2m}{2m}\geq&\frac{\left( 2m+r
 \right) ^{\frac14 \left( 2m+r \right) ^{2}}}{{r}^{\frac14{r}^{2}}\left( 2m \right) ^{{m}^{2}}{\e^{\frac32mr}}},
\end{align*}
and we conclude that
\begin{equation*}
\lnei{r+d+1}{d}\geq \frac{\left( r+d \right) ^{\left( \frac{r}{2}+\frac{d}{2} \right) ^{2}}}{{r}^{{(\frac{r}{2})}^{2}}{d}^{{(\frac{d}{2})}^{2}}{\e^{3\frac{r}{2}\frac{d}{2}}}}.
\end{equation*}
\end{proof}

The following corollary is a further simplification of the bound. It has the form~$n^{\frac{dn}{2}(1+o(1))}$ when $d$ is fixed and $n\rightarrow \infty$. 
\begin{corollary}
The number of labeled neighborly polytopes in even dimension $d$ with $n$ vertices fulfills 
\[\lnei{n}{d}\geq  \left( \frac{n-1}{\e^{3/2}}\right)^{\frac12(n-d-1)d}.\]
\end{corollary}
\begin{proof}
Since ${r}^{{(\frac{r}{2})}^{2}}{d}^{{(\frac{d}{2})}^{2}}\leq{(r+d)}^{{(\frac{r}{2})}^{2}+{(\frac{d}{2})}^{2}}$ we obtain
\[
\lnei{r+d+1}{d}\geq \frac{\left( r+d \right) ^{\left( \frac{r}{2}+\frac{d}{2} \right) ^{2}}}{{r}^{{(\frac{r}{2})}^{2}}{d}^{{(\frac{d}{2})}^{2}}{\e^{3\frac{r}{2}\frac{d}{2}}}}\geq  \frac{\left( r+d \right) ^{\frac{rd}{2}}}{{\e^{\frac{3{rd}}{4}}}}.\qedhere
\]
\end{proof}

Observe that this bound is not only useful for neighborly polytopes whose number of vertices is very large with respect to the dimension, but also for neighborly polytopes with fixed corank and large dimension.

A final observation is that we can translate these bounds for even dimensional neighborly polytopes to bounds for neighborly polytopes in odd dimension just by taking pyramids, because a pyramid over an even dimensional neighborly polytope is always neighborly. (If simpliciality was needed, any extension in general position of the dual of an even-dimensional neighborly polytope would work too.)

\begin{corollary}\belowdisplayskip=-12pt \label{cor:oddneighpoly}
The number of labeled neighborly polytopes in odd dimension $d$ with $n$ vertices fulfills 
\begin{align*}
\lnei{n}{d}\geq\lnei{n\!-\!1}{d\!-\!1}&\geq
\frac{\left(n-2\right) ^{\left( \frac{n-2}{2}\right) ^{2}}}{{(n-d-1)}^{{\left(\frac{n-d-1}{2}\right)}^{2}} {(d-1)}^{{(\frac{d-1}{2})}^{2}}{\e^{\frac{3(d-1)(n-d-1)}{4}}}}\\
&\geq  \left( \frac{n-2}{\e^{3/2}}\right)^{\frac12(n-d-1)(d-1)}.
\end{align*}
\end{corollary}

\subsection{Many non-realizable neighborly matroids}

Exactly the same reasoning that leads to the bounds in Theorem~\ref{thm:lblnei} can be applied to give lower bounds for non-realizable neighborly matroids. From now on, let 
$\lnr{n}{r}$ represent the number of labeled non-realizable neighborly oriented matroids of rank $r$ with $n$ elements.

\begin{theorem}\label{thm:nonrealizablebound}
 The number of labeled non-realizable neighborly oriented matroids of odd rank~$\rd$ with $n$ elements is at least
\begin{equation*}
\lnr{n}{s}\geq \frac{\left( n-1 \right) ^{\frac12{(\rd-5)(n-\rd)}}}{\left(\frac{n-\rd+4}{2}\right) ^{4}\e^{\frac34(\rd-5)(n-\rd)}}.
\end{equation*}
\end{theorem}
\begin{proof}[Proof sketch]
The principal observation is that an analogue of the inequality~\eqref{eq:lbbm} of Lemma~\ref{lem:lbbm} applies. That is, if $r\geq2$, $m\geq2$ and $n=2m+r+1$, then
\begin{equation*}
\lnr{n}{2m+1}\geq \lnr{n-2}{2m-1}\frac{n-1}{2}\lle{n-2}{r}.
\end{equation*}
This uses the Double Extension Theorem~\ref{thm:thethm} and the fact that all the lexicographic extensions of a non-realizable matroid are non-realizable.

Moreover, by Theorem~\ref{thm:nonrealizable}, $\lnr{r+5}{5}\geq 1$ for all $r\geq5$.
Which means that for $m\geq3$ we can mimic the proof of Theorem~\ref{thm:lblnei} to get 
\begin{align*}
\lnr{r+1+2m}{2m+1}\geq& \prod_{i=3}^{m} {\prod_{j=1}^{r}{ \left(2i+j\right)}}\\\geq&\exp\left(\int_{i=2}^{m} {\int_{j=0}^{r}{ \ln\left(2i+j\right) \mathrm{d}j} \mathrm{d}i}\right)\\
=&\frac{2^8\left( 2m+r \right) ^{\frac{\left( 2m+r \right) ^
{2}}{4}}}{ \left( r+4 \right) ^{\frac{\left( r+4 \right) ^{2}}{4}} {(2m)}^{{m}^{2
}}{\e^{\frac{3\left( m-2 \right)r}{2} }}}
\geq\frac{\left( 2m+r \right) ^{(m-2)r}}{\left(\frac{r+4}{2}\right) ^{4}\e^{\frac{3 \left( m-2 \right)r}{2} }}.\qedhere
\end{align*}
{}
\end{proof}

\medskip
\section*{Acknowledgements}

I would like to thank Uli Wagner for stimulating discussions that originated this research and Guillem Perarnau and Juanjo Ru\'e for fruitful conversations. Moreover, I am indebted to Julian Pfeifle for his advice, comments and corrections that greatly helped to improve the presentation. I also want to express my gratitude to an anonymous referee who provided a lot of useful observations and suggestions. 

\appendix

\section{Appendix}\label{sec:appendix}
This appendix contains the proofs of Propositions~\ref{prop:allquotientsofle} and~\ref{prop:allquotientsofGalesewnareGalesewn}, as well as some intermediate lemmas.

For the proof of Proposition~\ref{prop:allquotientsofle} we need a pair of results. 
The first one concerns inseparable elements, and shows the relation between circuits/cocircuits through~$x$ and circuits/cocircuits through~$y$ when $x$ and $y$ are inseparable. 

\begin{lemma}\label{lem:circinseparable}
 Let $\cM$ be a uniform oriented matroid with two $\alpha$-inseparable elements~$x$ and~$y$.
\begin{enumerate}
 \item For every circuit $X\in\ci(\cM)$ with $X(x)=0$ and $X(y)\neq 0$, there is a circuit $X'\in\ci(\cM)$ with $X'(x)=-\alpha X(y)$, $X'(y)=0$ and $X'(e)=X(e)$ for all $e\notin\{x,y\}$;
 \item For every cocircuit $C\in\co(\cM)$ with $C(x)=0$ and $C(y)\neq 0$, there is a cocircuit $C'\in\co(\cM)$ with $C'(x)=\alpha C(y)$, $C'(y)=0$ and $C'(e)=C(e)$ for all $e\notin\{x,y\}$.
\end{enumerate}

\end{lemma}
\begin{proof}
Both statements are equivalent by duality. We prove the first one.

Let $X'\in \ci(\cM)$ be the circuit with support $\underline X'=\underline X \setminus y \cup x$ and such that $X'({x})=-\alpha X(y)$. This circuit exists because $\cM$ is uniform. We will see that $X'({e})= X(e)$ for all $e\in \underline X'\setminus x$.
Let $e\in \underline{X}\setminus y$, and let $C$ be the cocircuit of~$\cM$ with $\underline C=E\setminus(\underline{X}\setminus y\setminus e)$. That makes $\underline C\cap \underline X=\{e,y\}$. Since $y$ and $x$ are $\alpha$-inseparable, $C(x)=-\alpha C(y)$, and by circuit-cocircuit orthogonality,
\[X(y)X(e)=-C(y)C(e)=\alpha C({x})C(e).\]
But $\underline C\cap \underline X'=\{e,x\}$, and hence, again by orthogonality, \(X'(x)X'(e)=-C({x})C(e).\)
The conclusion now follows from $X'({x})=-\alpha X(y)$.
\end{proof}

In this lemma, the hypothesis of uniformity is important, since the result does not hold in general.

The second lemma concerns the simultaneous contraction and deletion of $p$ and $a_1$.

\begin{lemma}\label{lem:contractdeletele}
If $\cM$ is uniform and  $p=[a_1^{\ep_1},\dots]$, then
 \[\cM/a_1 = ( \cM[p]\setminus p)/a_1= (\cM[p]\setminus a_1)/p .\]

\end{lemma}
\begin{proof}
The first equality is direct. The second one follows from Lemma~\ref{lem:circinseparable}. Indeed, every cocircuit of $( \cM[p]/a_1)\setminus p$ corresponds to a cocircuit~$C$ of~$\cM[p]$ with $C(a_1)=0$ and $C(p)\neq 0$. By Lemma~\ref{lem:circinseparable}, the
values of $C$ on $e\notin\{a_1,p\}$ coincide with the values of $C'$ on $e\notin\{a_1,p\}$, where $C'$ is a cocircuit of $\cM[p]$ with $C'(a_1)\neq 0$ and $C'(p)=0$. That is, $C'$ corresponds to a cocircuit of $( \cM[p]/p)\setminus a_1$.
\end{proof}

We are now ready to prove Proposition~\ref{prop:allquotientsofle}, and restate it here for the reader's convenience.

\smallskip
\noindent\textbf{Proposition \ref{prop:allquotientsofle}}
\emph{
Let $\cM$ be a uniform oriented matroid of rank $\rr$ on a ground set~$E$, and let 
$\cM[p]$ be the lexicographic extension of $\cM$ by $p=[a_1^{\ep_1},a_2^{\ep_2},\dots,a_\rr^{\ep_\rr}]$.
Then
\begin{align}
\cM[p]/p\ &\stackrel{\varphi}{\simeq}\ (\cM/a_1)[a_2^{-\ep_1\ep_2},\dots,a_\rr^{-\ep_1\ep_\rr}],\tag{\ref{eq:Mmodp}}\\
\cM[p]/a_i\ &=\ (\cM/a_i)[a_1^{\ep_1},\dots,a_{i-1}^{\ep_{i-1}},a_{i+1}^{\ep_{i+1}},\dots,a_\rr^{\ep_\rr}], \text{ and }\tag{\ref{eq:Mmoda}}\\
\cM[p]/e\ &=\ (\cM/e)[a_1^{\ep_1},a_2^{\ep_2},\dots,a_{\rr-1}^{\ep_{\rr-1}}];\tag{\ref{eq:Mmode}}
\end{align}
where $e\in E$ is any element different from $p$ and any $a_i$. The isomorphism $\varphi$ in~\eqref{eq:Mmodp} is $\varphi(e)=e$ for all $e\in E\setminus \{p,a_1\}$ and $\varphi(a_1)=[a_2^{-\ep_1\ep_2},\dots,a_\rr^{-\ep_1\ep_\rr}]$; the latter is the extending element.
}
\begin{proof}
The proof of \eqref{eq:Mmoda} and \eqref{eq:Mmode} is direct 
just by observing the signature of $p$ in~$\cM[p]$.

To prove \eqref{eq:Mmodp}, observe that $(\cM[p]/p)\setminus a_1=\cM/ a_1$ by Lemma~\ref{lem:contractdeletele}. Therefore, we only need to prove that the signature of the extension of $(\cM[p]/p)\setminus a_1$ by $a_1$ coincides with that of the lexicographic extension by $[a_2^{-\ep_1\ep_2},\dots,a_\rr^{-\ep_1\ep_\rr}]$. 
That is, let $C\in \co(\cM[p])$ be a cocircuit of $\cM[p]$ with $C(p)=0$ and $C(a_1)\neq 0$ and let $k>1$ be minimal with $C(a_k)\neq 0$. We want to see that $C(a_1)=-\ep_1\ep_kC(a_k)$. 

Because $a_1$ and $p$ are $(-\ep_1)$-inseparable, Lemma~\ref{lem:circinseparable} yields a cocircuit $C'\in \co(\cM[p])$ with $C'(p)=-\ep_1C(a_1)$ and $C'(a_1)=0$ and such that $k$ is minimal with $C'(a_k)\neq 0$. Moreover $C'(a_k)=C(a_k)$ and by the signature of the lexicographic extension $C'(p)=\ep_kC'(a_k)=\ep_kC(a_k)$. The claim follows from comparing both expressions for $C'(p)$.
\end{proof}

The proof of Proposition~\ref{prop:allquotientsofGalesewnareGalesewn} uses Proposition~\ref{prop:allquotientsofle},  Lemma~\ref{lem:quotientsofGalesewnareGalesewn} and Lemma~\ref{lem:galesewingorder} below to deduce that subpolytopes (convex hull of subsets of vertices) of Gale sewn polytopes are also Gale sewn. 

Lemma~\ref{lem:galesewingorder} shows that when Gale sewing, the roles of $a_1$, $p$ and $q$ can be exchanged.
Indeed, the isomorphism in~\eqref{eq:changepq} implies that we can switch the roles of $p$ and $q$, while the isomorphism in~\eqref{eq:changepa1} shows how $a_1$ can also be considered as one of the sewn elements.  

\begin{lemma}\label{lem:galesewingorder}
Let $\cM$ be a uniform oriented matroid on a ground set $E$, and consider the lexicographic extensions by
\begin{align*}
 p&=[a_1^{\ep_1},\dots,a_r^{\ep_r}],& q&=[p^-,a_1^{-},\dots,a_{r-1}^{-}];\\
 p'&=[a_1^{-\ep_1},\dots,a_r^{-\ep_r}],& q'&=[p'^-,a_1^{-},\dots,a_{r-1}^{-}];\\
 p''&=[a_1^{+},a_2^{-\ep_1\ep_2},\dots,a_r^{-\ep_1\ep_r}],& q''&=[p''^-,a_1^{-},\dots,a_{r-1}^{-}].
\end{align*}
Then
 \begin{align}
\cM[p][q]&\stackrel{\varphi}{\simeq}\cM[p'][q'],\text{ and }\label{eq:changepq}\\  \cM[p][q]&\stackrel{\psi}{\simeq}\cM[p''][q''],  \label{eq:changepa1} 
 \end{align}
where the bijection $\varphi:E\cup\{p,q\}\rightarrow E\cup\{p',q'\}$ is 
\[\varphi(p)=q',\, \varphi(q)=p'\text{ and }\varphi(e)=e\text{ for }e\in E;\]
and $\psi:E\cup\{p,q\}\rightarrow E\cup\{p'',q''\}$ is defined as
\[\psi(p)=\begin{cases}a_1 \text{ if }\ep_1=+,\\q''\text{ if }\ep_1=-\end{cases}\!\!\!\!\!\!,\,\,
\psi(q)=\begin{cases}q'' \text{ if }\ep_1=+,\\a_1\text{ if }\ep_1=- \end{cases}\!\!\!\!\!\!,\,\,
\psi(a_1)=p''\text{ and }
\psi(e)=e\text{ for }e\in E\setminus\{a_1\}.
\]
\end{lemma}
\begin{proof}
We start proving that $\cM[p][q]\stackrel{\varphi}{\simeq}\cM[p'][q']$. For every cocircuit $C\in\co(\cM[p][q])$ we want to find a cocircuit $C'\in\co(\cM[p'][q'])$ with $C'(\varphi(a))=C(a)$ for all $a\in E\cup\{p,q\}$. That is $C'(p')=C(q)$, $C'(q')=C(p)$ and $C'(e)=C(e)$ for $e\in E$. Let $D$ be the restriction of $C$ to~$E$.

If $C(q)\neq 0$ and $C(p)\neq 0$, let $i$ be minimal with $D(a_i)\neq 0$. By construction, $C(p)=\ep_iD(a_i)$ and $C(q)=-C(p)=-\ep_iD(a_i)$. By the definition of $\cM[p'][q']$, there is a cocircuit $C'$ that expands $D$, with $C'(p')=-\ep_iD(a_i)=C(q)$ and $C'(q')=-C(p')=\ep_iD(a_i)=C(p)$.

To deal with the case when $C(q)= 0$ or $C(p)= 0$, we use Proposition~\ref{prop:allquotientsofle} to see that $\cM[p][q]/p{\simeq}\cM[p'][q']/q'$ and $\cM[p][q]/q{\simeq}\cM[p'][q']/p'$.

To prove that $\cM[p][q]\stackrel{\psi}{\simeq}\cM[p''][q'']$ we assume that $\ep_1=+$ (otherwise use~\eqref{eq:changepq} to exchange $p$ with $q$).
 In this case we prove that \(\cM[p]\simeq\cM[p'']\), which implies~\eqref{eq:changepa1} because when $p$ and $a_1$ are $(-1)$-inseparable the lexicographic extensions by $[p^-,a_1^{-},\dots,a_{r-1}^{-}]$ and $[a_1^-,p^-,\dots,a_{r-1}^{-}]$ coincide.

For every cocircuit $C\in\co(\cM[p])$ we want to find a cocircuit $C''\in\co(\cM[p''])$ with $C''(p'')=C(a_1)$, $C''(a_1)=C(p'')$ and $C''(e)=C(e)$ for $e\in E\setminus\{a_1\}$. Again, let $D$ be the restriction of $C$ to $E$.

If $C(p)\neq 0$ and $C(a_1)\neq 0$, then $C(p)=C(a_1)=D(a_1)$. Moreover, $D$ is also expanded to a cocircuit $C''$ of $\cM[p'']$ with $C''(p'')=C''(a_1)=D(a_1)$. For circuits with $C(a_1)=0$, observe that $\cM[p]/a_1=\cM[a_2^{\ep_2},\dots,a_r^{\ep_r}]\simeq\cM[p'']/p''$ by Proposition~\ref{prop:allquotientsofle}. Finally, if $C(p)= 0$ then, again by Proposition~\ref{prop:allquotientsofle}, $\cM[p]/p\simeq\cM[a_2^{-\ep_2},\dots,a_r^{-\ep_r}]=\cM[p'']/a_1$.\end{proof}

 With this lemma we have the last ingredient needed to prove that all the subpolytopes of a Gale sewn polytope are Gale sewn.

\smallskip
\noindent\textbf{Proposition \ref{prop:allquotientsofGalesewnareGalesewn}}
\emph{
 If $P$ is a neighborly polytope in $\cG$, and $a$ is a vertex of~$P$, then $Q=\conv(\verts(P)\setminus a)$ is also a neighborly polytope in $\cG$.
}
\begin{proof}
Let $\cP$ be the oriented matroid of $P$ and $e$ the element of $\cP$ corresponding to the vertex $a$. Observe that $\cP\setminus e$ is the oriented matroid of~$Q$.
The proof is by induction on the rank of $\cP$. When $\cP$ has rank $0$ then $\Gale\cP\simeq\{e_1, \dots, e_r, -\sum_{i=1}^r e_i\}$ and $\pGale{\cP\setminus e}\simeq \{e_1, \dots, e_r, -\sum_{i=1}^r e_i\}/e\simeq\{e_1, \dots, e_{r-1}, -\sum_{i=1}^{r-1} e_i\}$.

Otherwise, let $\cM=\Gale\cP$. That $\cP$ belongs to $\cG$ means that there is a matroid $\cN$ whose dual~$\Gale\cN$ is in $\cG$, such that $\cM=\cN[p][q]$ where $p=[a_1^{\ep_1},a_2^{\ep_2},\dots,a_r^{\ep_r}]$ and $q=[p^-,a_1^{-},\dots,a_{r-1}^{-}]$.

We will prove that for every $e\in\cP$, there is some $\tilde e\in \cN$ fulfilling
 \begin{equation}\label{eq:contracte}\pGale{\cP\setminus e}\simeq (\cN/ \tilde e)[\tilde p][\tilde q] =\pGale{\Gale\cN\setminus \tilde e}[\tilde p][\tilde q],\end{equation} 
where $\tilde p=[\tilde a_1^{\tilde \ep_1},\tilde a_2^{\tilde \ep_2},\dots,\tilde a_r^{\tilde \ep_r}]$ and $\tilde q=[\tilde p^-,\tilde a_1^{-},\dots,\tilde a_{r-1}^{-}]$ for some $\tilde a_i$'s and~$\tilde \ep_i$'s. Since $\rank\left( \Gale\cN\right)=\rank\left( \cP\right) -2$, by the induction hypothesis $(\Gale\cN\setminus e')\in \cG$ and our claim follows directly from \eqref{eq:contracte}. 

 If $e=q$, then by Lemma~\ref{lem:quotientsofGalesewnareGalesewn} we know that $\pGale{\cP\setminus e}=\left(\cN[p][q]\right)/q\simeq\left(\cN/a_1\right)[\tilde p][\tilde q]$, where $\tilde p=[a_2^{-\ep_1\ep_2},\dots,a_r^{-\ep_1\ep_r}]$ and $\tilde q=[\tilde p^{-},\dots,a_{r-1}^{-}]$. The case $e=p$ is analogous because of Lemma~\ref{lem:galesewingorder}.
 If $e=a_i$, then $\pGale{\cP\setminus e}\simeq\left(\cN/a_i\right)[\tilde p][\tilde q]$, 
where $\tilde p$ and $\tilde q$ have the same signature as $\tilde p$ and $\tilde q$ but omitting the element $a_i$.
 For the remaining elements $e$, $\pGale{\cP\setminus e}\simeq\left(\cN/e\right)[\tilde p][\tilde q]$.
\end{proof}

\bibliographystyle{plain}
\bibliography{ManyNeighborly} 
\end{document}